\documentclass[11pt]{amsart}
\usepackage{float}
\usepackage{soul}
\usepackage{wrapfig}
\usepackage{graphicx} 
\usepackage{caption}
\usepackage{subcaption}
\usepackage[backref]{hyperref}
\usepackage{cleveref}

\usepackage{diagbox}
\usepackage{booktabs}
\usepackage{amssymb}
\usepackage{amsthm}
\usepackage{amsmath}
\usepackage{amsxtra}
\usepackage{mathrsfs}
\usepackage{hyperref}
\usepackage{epsfig}
\usepackage{comment}
\usepackage{enumerate}
\usepackage{enumitem}
\setlist{itemsep=4pt}
\usepackage{fullpage}
\usepackage{hyperref}
\hypersetup{colorlinks=true,urlcolor=blue,citecolor=blue,linkcolor=blue}
\usepackage{colonequals}
\usepackage{booktabs}
\usepackage{multirow}
\usepackage{numprint}
\npdecimalsign{.}
\usepackage{booktabs}
\usepackage{tikz-cd}
\usepackage{arydshln}
\usepackage{wasysym}

\usepackage{mathtools}
\usepackage{tikz}
\usetikzlibrary{positioning,math,decorations.markings,arrows,shapes,calc}
\usepackage{setspace}

\usepackage{pgfplots}
\pgfplotsset{compat=1.15}
\usepackage{mathrsfs}
\usetikzlibrary{arrows}

\newcommand{\NW}{45}

\newcommand{\NE}{135}

\newcommand{\SE}{225}

\newcommand{\SW}{315}
\newcommand{\dist}{-0.5}
\newcommand{\distlarger}{-0.65}

\newtheorem{theorem}{Theorem}[section]

\theoremstyle{plain}
\newtheorem{thm}[theorem]{Theorem}
\newtheorem{prop}[theorem]{Proposition}
\newtheorem{lem}[theorem]{Lemma}

\newtheorem{conj}[theorem]{Conjecture}

\theoremstyle{definition}

\newtheorem{remark}[theorem]{Remark}

\newcommand{\Q}{\mathbb{Q}}
\newcommand{\Z}{\mathbb{Z}}

\newcommand{\F}{\mathbb{F}}

\renewcommand{\geq}{\geqslant}
\renewcommand{\leq}{\leqslant}

\def\Alphabet{A,B,C,D,E,F,G,H,I,J,K,L,M,N,O,P,Q,R,S,T,U,V,W,X,Y,Z}
\def\alphabet{a,b,c,d,e,f,g,h,i,j,k,l,m,n,o,p,q,r,s,t,u,v,w,x,y,z}
\def\endpiece{xxx}
\def\makeAlphabet[#1]{\expandafter\makeA#1,xxx,}
\def\makealphabet[#1]{\expandafter\makea#1,xxx,}
\def\makeA#1,{\def\temp{#1}\ifx\temp\endpiece\else%
\mkbb{#1}\mkfrak{#1}\mkbf{#1}\mkcal{#1}\mkscr{#1}\mkbs{#1}\expandafter\makeA\fi}%
\def\makea#1,{\def\temp{#1}\ifx\temp\endpiece\else\mkfrak{#1}\mkbf{#1}\mkbs{#1}\expandafter\makea\fi}%
\def\mkbb#1{\expandafter\def\csname bb#1\endcsname{\mathbb{#1}}}
\def\mkfrak#1{\expandafter\def\csname fr#1\endcsname{\mathfrak{#1}}}
\def\mkbf#1{\expandafter\def\csname b#1\endcsname{\mathbf{#1}}}
\def\mkcal#1{\expandafter\def\csname c#1\endcsname{\mathcal{#1}}}
\def\mkscr#1{\expandafter\def\csname s#1\endcsname{\mathscr{#1}}}
\def\mkbs#1{\expandafter\def\csname bs#1\endcsname{{\boldsymbol{#1}}}}
\def\makeop[#1]{\xmakeop#1,xxx,}
\def\mkop#1{\expandafter\def\csname #1\endcsname{{\mathrm{#1}}}} %
\def\xmakeop#1,{\def\temp{#1}\ifx\temp\endpiece\else\mkop{#1}\expandafter\xmakeop\fi}%
\def\makeup[#1]{\xmakeup#1,xxx,}
\def\mkup#1{\expandafter\def\csname #1\endcsname{{\mathrm{#1}\,}}} %
\def\xmakeup#1,{\def\temp{#1}\ifx\temp\endpiece\else\mkup{#1}\expandafter\xmakeup\fi}%
\makeAlphabet[\Alphabet]
\makealphabet[\alphabet]
\makeop[Hom,Tor,Sel,GL,H,R,ord,Pic,NS,Br,BR,Zar,fppf,rig,cris,syn,dR,Gal,Ker,Im,Coker,ker,coker,BSD,Res,T,Frob,rk,sh,nr,Re,Im,im,Jac,red]
\makeup[Spec,Proj,Spwf,Sp,Sh,Spf,Sch,id,Tot,sm,an,tor,holim,hocolim,dim,red,rank,rk,CH,Tr,GM]

\usepackage{xcolor}
\definecolor{firebrick4}{rgb}{0.84,0.1,0.35}
\definecolor{ududff}{rgb}{0.30196078431372547,0.30196078431372547,1.}

\newcommand{\isom}{\cong}

\title{Computing quadratic points on modular curves $X_0(N)$}

\begin{document}


\author[Adzaga]{Nikola Ad\v zaga}
\address{Nikola Ad\v zaga, Department of Mathematics, Faculty of Civil Engineering, University of Zagreb, Fra Andrije Kačića-Miošića 26, 10 000 Zagreb, Croatia}
\email{\url{nadzaga@grad.hr}}

\author[Keller]{Timo Keller}
\address{Timo Keller, Leibniz Universität Hannover, Institut für Algebra, Zahlentheorie und Diskrete Mathematik, Welfengarten 1, 30167 Hannover, Germany}
\email{\url{keller@math.uni-hannover.de}}
\urladdr{\url{https://www.timo-keller.de}}

\author[Michaud-Jacobs]{Philippe Michaud-Jacobs}
\address{Philippe Michaud-Jacobs, Mathematics Institute, University of Warwick, CV4 7AL, United Kingdom}
\email{\url{p.rodgers@warwick.ac.uk}}
\urladdr{\url{https://warwick.ac.uk/fac/sci/maths/people/staff/michaud/}}

\author[Najman]{Filip Najman}
\address{Filip Najman, Department of Mathematics, Faculty of Science, University of Zagreb, Bijenička cesta 30, 10000 Zagreb, Croatia}
\email{\url{fnajman@math.hr}}
\urladdr{\url{https://web.math.pmf.unizg.hr/~fnajman/}}

\author[Ozman]{Ekin Ozman}
\address{Ekin Ozman, Bogazici University,
Department of Mathematics,
Bebek, Istanbul, 34342, 
Turkey}
\email{\url{ekin.ozman@boun.edu.tr}}

\author[Vukorepa]{Borna Vukorepa}
\address{
  Borna Vukorepa, Department of Mathematics,
  Faculty of Science,
  University of Zagreb,
  Bijenicka cesta 30,
  10000 Zagreb,
  Croatia
}
\email{\url{borna.vukorepa@gmail.com}}

\thanks{T.\@ K.\@ was supported by the Deutsche Forschungsgemeinschaft (DFG), Projektnummer STO 299/18-1, AOBJ: 667349 while working on this article. P.\@ M.\@ is supported by an EPSRC studentship EP/R513374/1 and has previously used the surname Michaud-Rodgers. F.\@ N.\@  and B.\@ V.\@  are supported by QuantiXLie Centre of Excellence, a
  project co-financed by the Croatian Government and European Union
  through the European Regional Development Fund - the Competitiveness
  and Cohesion Operational Programme (Grant KK.01.1.1.01.0004). E.\@ O.\@ is partially supported by TUBITAK Project No 122F413. }
\date{\today}

\begin{abstract}
    In this paper we improve on existing methods to compute quadratic points on modular curves and apply them to successfully find all the quadratic points on all modular curves $X_0(N)$ of genus up to $8$, and genus up to $10$ with $N$ prime, for which they were previously unknown. The values of $N$ we consider are contained in the set
    $$\mathcal{L}=\{58, 68, 74, 76, 80, 85, 97, 98, 100, 103, 107, 109, 113, 121, 127 \}.$$

    We obtain that all the non-cuspidal quadratic points on $X_0(N)$ for $N\in \mathcal{L}$ are CM points, except for one pair of Galois conjugate points on $X_0(103)$ defined over $\Q(\sqrt{2885})$. We also compute the $j$-invariants of the elliptic curves parametrised by these points, and for the CM points determine their geometric endomorphism rings. 
\end{abstract}


\maketitle

\section{Introduction}\label{intro_section}

A fundamental problem in arithmetic geometry, with many consequences in the theory of elliptic curves and to Diophantine equations, is to understand the rational points and low-degree points on modular curves. In two seminal papers, Mazur determined the rational points on the two most well-known families of modular curves: $X_1(N)$ in 1977 \cite{eisenstein} and $X_0(p)$ for prime $p$ in 1978 \cite{mazur:isogenies}. Kenku completed the determination of rational points for all $X_0(N)$ in 1981 (see \cite{kenku:125} and the references therein).

In the decades since, all degree $d$ points on $X_1(N)$ have been determined for $d=2$ by Kamienny, Kenku and Momose \cite{kamienny92, KM88}, for $d=3$ by Derickx, Etropolski, van Hoeij, Morrow, and Zureick-Brown \cite{Deg3Class}, and for $4\leq d\leq 7$ and prime values of $N$ by Derickx, Kamienny, Stein, and Stoll \cite{DKSS}. For general $d$, Merel proved that there exists an explicit constant $C_d$, which depends only on $d$, such that when $N>C_d$  the only points of degree $\leq d$ on $X_1(N)$ are cusps \cite{merel}.

The degree $d$ points on all $X_0(N)$ have not been determined for any $d>1$. One major difficulty is that, as opposed to $X_1(N)$, the modular curve $X_0(N)$ has non-cuspidal quadratic points for infinitely many $N$, coming from elliptic curves with complex multiplication (CM). The existence of these CM points prevents the methods used to determine the degree $d$ points on $X_1(N)$ to be used on $X_0(N)$.

However, if one ignores the CM points, then it is expected that there should be only finitely many $N$ such that $X_0(N)$ has non-cuspidal non-CM points. The following conjecture is widely believed. 

\begin{conj}[Quadratic isogenies conjecture] \label{conj:qic}
There exists an integer $C$ such that if $K$ is a quadratic field and $N>C$ is an integer, then any $P\in X_0(N)(K)$ is either a cusp or a CM-point. 
\end{conj}

 We emphasize that the constant $C$ in this conjecture does not depend on the quadratic field $K$. \Cref{conj:qic} implies the conjecture of Elkies which states that the set of $N$ such that $X^+_0(N)(\Q)$ contains points that are neither CM nor cusps is finite (see \cite{Elkies2004}).

While there are no general bounds as in \Cref{conj:qic}, there has been much more success in classifying all the quadratic points on $X_0(N)$ for fixed values of $N$. Bruin and Najman described all the quadratic points on hyperelliptic $X_0(N)$ such that the Jacobian $J_0(N)$ has rank $0$ over $\Q$ \cite{bruin-najman}.  Ozman and Siksek then determined all the quadratic points on the non-hyperelliptic $X_0(N)$ with $J_0(N)$ having rank $0$ over $\Q$ and with genus $\leq 5$  \cite{ozman-siksek}. Box determined the quadratic points on all the $X_0(N)$ with  genus $\leq 5$ and with $J_0(N)$ having positive rank over $\Q$ \cite{box21}. Najman and Vukorepa described all the quadratic points on the bielliptic $X_0(N)$ (for which this had not already been done in previous work) \cite{najman-vukorepa}. We provide a more detailed overview of all known results in Section \ref{sec:overview}.

In this paper we complete the determination of all the quadratic points on the modular curves $X_0(N)$ of genus up to $8$, and for $X_0(N)$ with $N$ prime of genus up to $10$. The motivation for this is threefold. Firstly, the results (i.e. the tables with explicit quadratic points) are useful as they have applications to Diophantine equations. In order to apply the so-called `Modular Approach' over a given real quadratic field $K$ (see for instance \cite{Freitas-Siksek} for details), one
requires
the irreducibility of the mod $p$ representation
of a Frey elliptic curve defined over $K$. This Frey elliptic curve
often has a $2$ or $3$-isogeny defined over $K$ too. Therefore, if the mod $p$
representation is reducible, then the
Frey curve gives a point on $X_0(2p)(K)$
or $X_0(3p)(K)$. Thus having a complete understanding 
of quadratic points is useful in establishing irreducibility
for small values of $p$. These type of results have been useful for instance in \cite{Freitas-Siksek,Khawaja-Jarvis,Michaud-Jacobs}.

Secondly, another motivation for these computations is to give evidence for \Cref{conj:qic} by showing that as $N$ (and hence the genus of $X_0(N)$) gets larger, the non-cuspidal non-CM quadratic points become rarer. 

Finally, a further motivation was to implement state-of-the-art techniques for determining quadratic points in a general and efficient way. The modular curves we study provide a natural testing ground for this.

Let 
\begin{equation}\label{setL} \mathcal{L} \colonequals \left\{ 58,68,74,76,80,85, 97, 98,100, 103, 107, 109, 113, 121, 127\right\}.\end{equation}
Our main result is the complete determination of the non-cuspidal quadratic points on $X_0(N)$ for $N\in \mathcal{L}$. 
\begin{thm} \label{maintm}
Let $N \in \mathcal{L}$. The finitely many non-cuspidal quadratic points on the curve $X_0(N)$ are displayed in the tables in \Cref{sec:tables}.
\end{thm}

We note that every quadratic point on $X_0(N)$ for $N \in \mathcal{L}$ is either a cusp or CM point, apart from a pair of quadratic points defined over the field $\Q(\sqrt{2885})$ on the curve $X_0(103)$.

 For all $N \in \mathcal{L}$ the genus of $X_0(N)$ is between $6$ and $10$. Together with the previous results of \cite{bruin-najman, ozman-siksek, box21, BPN, najman-vukorepa, balcik, Vux91} this completes the classification of quadratic points on the modular curves $X_0(N)$ of genus up to $8$, and the classification of quadratic points on $X_0(N)$ of genus up to $10$ with $N$ prime.
 
We use three methods to determine the quadratic points on these curves: 
\begin{itemize}
    \item The `going down' method,
    \item The `rank $0$' method,
    \item The `Atkin--Lehner sieve' method.
\end{itemize}
 The `going down' method can be applied in certain cases when the quadratic points on $X_0(M)$, for some proper divisor $M$ of $N$, have been previously determined. We can often reduce the problem to determining the rational points on several Atkin--Lehner quotients of several modular curves, and checking whether finitely many lifts of quadratic points on $X_0(M)$ lift to quadratic points on $X_0(N)$. More details about the `going down' method are given in \Cref{sec:gd}.  The `rank $0$' method, which is explained in more detail in \Cref{sec:rank0}, can be used when the Mordell--Weil group $J_0(N)(\Q)$ has rank $0$. It uses a type of Mordell--Weil sieve to determine the quadratic points on $X_0(N)$. The `Atkin--Lehner sieve' method is the most involved of the three methods. It uses a Mordell--Weil sieve involving an Atkin--Lehner involution on the curve $X_0(N)$, together with a symmetric Chabauty criterion. This method is based on ideas of Siksek \cite{siksek:symchabauty}, which were later further developed in \cite{box21, najman-vukorepa}. A successful application of the sieve reduces the problem to considering fixed points of an Atkin--Lehner involution and the rational points on a given Atkin--Lehner quotient. A more detailed description of this method can be found in \Cref{sec:sieve}.

Although the methods we use are similar to previously used methods, a major reason why we are able to extend these methods to consider curves of higher genus than were previously studied is thanks to the new techniques we introduce (in Section \ref{models_section}) to work with models and maps of curves $X_0(N)$ and their quotients. In particular, we work with models of $X_0(N)$ on which the action of the Atkin--Lehner involutions is simultaneously diagonalised, allowing us to easily compute quotient maps and see the relationships between quadratic points.  We also improve on existing methods for computing equations for the $j$-map, as well as introduce a fast method for testing nonsingularity at a given prime. We also note that our algorithms offer significant computational speed-ups when applied to certain previously computed levels (see Remark \ref{running_time_rem}).

The \texttt{Magma} \cite{magma} code used to support and verify the computations in this paper is available at \begin{center}
\url{https://github.com/TimoKellerMath/QuadraticPoints}
\end{center}

\subsection*{Acknowledgements} We would like to thank the organisers of the March 2022 \emph{Modular curves workshop} at MIT (hybrid) for brining the authors together and providing a starting point for this project. We would also like to thank Jennifer Balakrishnan, Shiva Chidambaram, Pip Goodman, David Holmes, Jeremy Rouse, Samir Siksek, Michael Stoll, and Damiano Testa for useful discussions. We thank the anonymous referee for helpful comments.

\section{Equations for \texorpdfstring{$X_0(N)$}{X0N} and associated maps}\label{models_section}

One of the main barriers to extending the work of classifying quadratic points on the curves $X_0(N)$ arises from the computational difficulties of working with curves of high genus. In order to use each of our three methods to compute quadratic points, it will be imperative to work with suitable models of the curves $X_0(N)$ and its quotients. In this section, we describe how to obtain such models, maps to Atkin--Lehner quotients, and equations for the $j$-map. We also describe a method for verifying whether a prime is of good or bad reduction for a given model.

The methods and techniques we present in this section are those required for the computations in this paper, but only make up a subset of the totality of the code that is available in the accompanying \texttt{Magma} files. The code also contains functions that can be used to compute maps to curves $X_0(M)$ with $M \mid N$, modular parametrisation maps, quotients by subgroups of the Atkin--Lehner group, and more.

\subsection{Models of $X_0(N)$}

Let $N$ be such that $X_0(N)$ is non-hyperelliptic of genus $g \geq 2$. A smooth model for the curve $X_0(N)$ in $\mathbb{P}^{g-1}$ may be obtained as the image of the canonical embedding on a basis of cusp forms in $S_2(N)$. This (now standard) process for obtaining models is described in detail in  \cite[pp.~17--38]{galbraith}, and the \texttt{Magma} code we used to do this is adapted from \cite{ozman-siksek}.

Although any basis of cusp forms for $S_2(N)$ may be used to obtain a model of $X_0(N)$, choosing certain bases can produce better models for our purposes. We work with \emph{diagonalised bases} of cusp forms, which give rise to \emph{diagonalised models} for $X_0(N)$. We now describe what these are. For $ M \mid N$ such that $\gcd(M,N/M)=1$, we denote by $w_M$ the corresponding Atkin--Lehner involution. The set of all Atkin--Lehner involutions forms an abelian $2$-group, which we denote $W$. We fix a basis $\{h_1, \dots, h_g\}$ for $S_2(N)$ and we may then
represent the elements of $W$ as matrices acting on $S_2(N)$ with respect to this basis. The set of all matrices in $W$ is simultaneously diagonalisable, meaning there exists a matrix $T$ such that $T w T^{-1}$ is a diagonal matrix for each $w \in W$. Applying this change of basis matrix to the basis $\{h_1, \dots, h_g\}$ produces a new basis $\{f_1, \dots, f_g\}$ of $S_2(N)$, for which the Atkin--Lehner involutions now act diagonally. That is, for each $w \in W$ and $1 \leq i \leq g$, we have $w(f_i) = \delta_{w,i} f_i$, where $\delta_{w,i} = \pm 1$.

After obtaining a diagonalised basis, we may then obtain a model for $X_0(N)$ in $\mathbb{P}^{g-1}_{x_1, \dots, x_g}$ via the image of the canonical embedding. The Atkin--Lehner involutions on  this model then act as they do on the $f_i$, namely $w(x_1 : \dots : x_g) = (\delta_{w,1}x_1: \dots : \delta_{w,g}x_g)$. 

There are three main reasons for working with diagonalised models. First and foremost, we will make use of the maps from $X_0(N)$ to its Atkin--Lehner quotients, and these maps will be straightforward to compute and work with, as we see in the following subsection. Secondly, we found that working with a diagonalised basis allowed us to quickly compute the image of the canonical embedding and produce equations for $X_0(N)$ with small coefficients. Finally, we found that the diagonalised models we produced usually had good reduction at small primes not dividing $2N$, which is important in the application of the sieving methods we use. We note that a diagonalised model of a non-split Cartan curve is used in \cite[pp.~254--255]{MJ_cart}, and that diagonalised bases of cusp forms for $N = p$ prime are used in \cite{najman-vukorepa}. In each of these cases, there is a single involution.

\subsection{Maps to Atkin--Lehner quotients}

Let $\{f_1, \dots, f_g\}$ be a diagonalised basis for $S_2(N)$ and assume that we have obtained the corresponding diagonalised model for $X_0(N)$ in $\mathbb{P}^{g-1}_{x_1, \dots, x_g}$ on which each Atkin--Lehner involution $w \in W$ acts as $w(x_i) = \delta_{w,i}x_i$, with  $\delta_{w,i} = \pm 1$. For each $w \in W$, we wish to compute a model for the quotient curve $X_0(N) / w$, as well as a map from $X_0(N)$ to this quotient. For each $N$ and $w$ we considered, the genus of $X_0(N) / w$ was $\geq 2$.

\subsubsection*{Case 1: The quotient $X_0(N) / w$ is non-hyperelliptic.} In this case, since the genus of $X_0(N) / w$ is $\geq 3$, the curve $X_0(N) / w$ may be canonically embedded. We choose the indices $i$ for which  $w(f_i) = f_i$, which we denote by $i_1, \dots, i_t$, and compute the image of the canonical embedding on the cusp forms $\{f_{i_1}, \dots, f_{i_t}\}$. This gives a model for $X_0(N) / w$ in $\mathbb{P}^{t-1}$, and the map from $X_0(N)$ to $X_0(N) / w$ is simply given by the projection map $(x_1 : \dots : x_g) \mapsto (x_{i_1} : \dots : x_{i_t})$.

\subsubsection*{Case 2: The quotient $X_0(N) / w$ is hyperelliptic.} In this case the curve $X_0(N) / w$ cannot be canonically embedded. Instead, we choose a set of indices $1 \leq i_1, \dots, i_t \leq g$ for which either $\delta_{w,i_j} = 1$ for $1 \leq j \leq t$, or  $\delta_{w,i_j} = -1$ for $1 \leq j \leq t$. In this way, the coordinates $x_{i_1}, \dots, x_{i_t}$ form a subset of coordinates on which $w$ acts trivially in projective space. We then project onto these coordinates and consider the image of this projection map. If this image is a curve of the expected genus (or if the projection map is of degree $2$ onto its image), then we have obtained a projective model for $X_0(N) / w$ in $\mathbb{P}^{t-1}$. If this is not the case, then we have in fact obtained a model for a (non-trivial) quotient of $X_0(N) / w$, in which case we can try to repeat this process with a different set of indices $1 \leq i_1, \dots, i_t \leq g$.  If we succeed in obtaining a model for $X_0(N) / w$, we may then apply a transformation to take this model to a standard model for this hyperelliptic curve in weighted projective space. We note that choosing a larger set of admissible indices $\{i_1, \dots, i_t\}$ increases the likelihood of success, but also gives a more complicated quotient map to the standard hyperelliptic model. 

The method presented in Case 2 succeeded for all values of $N$ and Atkin--Lehner quotients we considered. However, if no suitable set of indices can be found, then it is also possible to use \texttt{Magma}'s inbuilt \texttt{CurveQuotient} function (which we found to be slower and produced more complicated maps in the cases we tested).

\subsection{Equations for the $j$-map}

Assume we have a model for $X_0(N)$ obtained as the image of the canonical embedding on the cusp forms $\{f_1, \dots, f_g\} \in S_2(N)$. We describe how to obtain equations for the $j$-map, $j\colon X_0(N) \rightarrow X_0(1) \isom \mathbb{P}^1$. Obtaining these equations will allow us to compute the cusps on the models of our curves $X_0(N)$, compute the $j$-values of any quadratic points we compute, and finally compute quadratic points with a given $j$-invariant obtained using the `going down' method. The methods we present here are in fact applicable in a wider setting. In particular, they can be used to compute modular parametrisation maps, as well as maps to modular curves $X_0(M)$ with $M \mid N$.

We start by following the method outlined in \cite[pp.~2464--2466]{ozman-siksek}. As a modular function, we have the $q$-expansion \[j(q) = \frac{1}{q} + 744 + 196884q + 21493760q^2 + \cdots.\]  Using linear algebra, we compute homogeneous polynomials $F, G \in \Q[X_1, \dots, X_g]$ of the same degree $r$, such that \begin{equation}\label{eqs_j} \frac{F(f_1(q), \dots, f_g(q))}{G(f_1(q), \dots, f_g(q))} = j(q)\end{equation} up to some precision, say $O(q^{5N})$. We then need to verify that this equality holds up to arbitrary precision. At this point, one would usually attempt to verify this by checking the equality up to the Sturm bound, but since $j(q)$ is not a modular form (i.e. not holomorphic), this is not possible. 

One option is to use the fact that $j(q) = E_4(q)^3 / \Delta(q)$ and consider $F\Delta - GE_4^3$, which is a modular form of weight $12 + \deg(F)$. The Sturm bound will be quite large in this case, and this method will not generalise. Instead, applying the argument in \cite[p.~2466]{ozman-siksek}, it is enough to show that the equality (\ref{eqs_j}) holds up to precision $O(q^{d_{F/G}+d_j + 1})$. Here, $d_{F/G}$ and $d_j$ are the degrees of $F/G$ and $j$ respectively as rational functions on $X_0(N)$. The degree $d_j$ can be computed solely in terms of $N$. In \cite[p.~2466]{ozman-siksek}, the degree $d_{F/G}$ is explicitly computed, and once it has been verified that $d_{F/G} = d_j$, it is enough to check the equality (\ref{eqs_j}) up to precision $O(q^{2d_j+1})$. However, computing $d_{F/G}$ is slow, and computationally infeasible for curves of high genus that we need to work with. Instead, we will find an upper bound on $d_{F/G}$.

\begin{lem}
Let $F/G$ be a rational function on a canonically embedded curve $X \subset \mathbb{P}^{g-1}$, with $F$ and $G$ homogeneous polynomials of degree $r$. Then \[ \deg(F/G) \leq (2g-2) \cdot r. \]
\end{lem}

\begin{proof}
We have that \[ \deg(F/G) = \deg(\mathrm{div}_0(F/G)) \leq \deg(\mathrm{div}(F)), \] where $\mathrm{div(F)}$ and $\mathrm{div}_0(F)$ denote the divisor of $F$ and the divisor of zeros of $F$ respectively. Then \[ \mathrm{div}(F) = \mathrm{div}( X \cap \{F = 0\}). \] This is the intersection of a curve and a hypersurface. By B\'ezout's theorem (as stated in \cite[pp.~167--168]{shafarevich} for example), the degree of this intersection divisor is $\deg(X) \cdot  \deg(\{F = 0 \}) = (2g-2) \cdot r$. 
\end{proof}

Thanks to this lemma and using the formula $d_j = \prod_{p \mid N}\big(1+\frac{1}{p}\big)$, it suffices to check (\ref{eqs_j}) up to precision $O(q^m)$, where \[ m = \left((2g-2) \cdot r \right) + 1 + N \cdot \prod_{p \mid N}\Big(1+\frac{1}{p}\Big).   \]  

The usual way of finding suitable homogeneous polynomials $F$ and $G$ of the same degree $r$ is to test values $r = 2, 3, 4, \dots$ until a degree that works is reached. Since $r$ can be somewhat large, this can be very slow. In order to start at a suitable value of $r$, we use the following lemma, communicated to us by Jeremy Rouse. We note that this lemma aids us in choosing a value of $r$ that will work from the outset, but its correctness is in fact irrelevant for our computations.

\begin{lem}
Let $X$ be a curve of genus $g \geq 2$ and let $\varphi\colon X \rightarrow \mathbb{P}^1$ be a map of degree $d$. Define $r$ to be the smallest positive integer larger than $\frac{d}{2(g-1)} + \frac{1}{2}$. Then $\varphi$ may be expressed as a ratio of two holomorphic differential $r$-forms on $X$.
\end{lem}

In our set-up, any $r$-fold product of the cusp forms $f_i$ (which is a modular form of weight $2r$) corresponds to a holomorphic differential $r$-form  on $X_0(N)$, and so we may seek polynomials $F$ and $G$ of degree $r$, as defined in this lemma.

\begin{proof}   Write $\mathrm{div}(\varphi) = D_0 - D$, where $D_0$ and $D$ are effective divisors of degree $d$ (the divisors of zeros and poles of $\varphi$ respectively). We aim to show that there exists a holomorphic differential $r$-form, $h$, such that $f = \varphi h$ is also a holomorphic differential $r$-form. Then $\varphi = f/h$  as required. It will suffice to find a holomorphic form $h$ such that $\mathrm{div}(h) \geq D$. We define (for any divisor $D'$ on $X$), \[ \mathcal{L}^{(r)}(D') = \{ \text{meromorphic differential $r$-forms, } \mu~ |~ \mathrm{div}(\mu) \geq -D'\}. \] It will be enough to prove that $\dim(\mathcal{L}^{(r)}(-D)) > 0$. Let $K$ be a canonical divisor on $X_0(N)$. Then by \cite[p. 238]{miranda}, we have an isomorphism of vector spaces $ \mathcal{L}^{(r)}(-D) \cong \mathcal{L}(-D+rK). $ By Riemann--Roch, \begin{align*} \dim(\mathcal{L}(-D+rK)) &\geq \deg(-D+rK) - g + 1  \\ &  = -d + r(2g-2) -g +1. 
 \end{align*} This is positive if and only if $r > \frac{g + d -1}{2(g-1)}  = \frac{d}{2(g-1)} + \frac{1}{2}$, as required.
\end{proof}

We note that the value of $r$ obtained by applying this lemma is not necessarily the minimal $r$ one can choose. However, in the cases we tested, we found it was usually the minimum possible value, or very close to it.

\subsection{Testing nonsingularity}

In order to apply the `rank $0$' and `Atkin--Lehner sieve' methods we use for computing quadratic points, it is crucial to be able to verify whether or not, for a given prime $p$, the reduction of a model for $X_0(N)$ mod $p$ is singular. If the model is singular mod $p$, then \texttt{Magma}'s inbuilt \texttt{IsSingular} command will determine this right away. However, if the model is nonsingular mod $p$, then for curves of genus $ \geq 9$ we found it was very slow to check this directly. Instead we use the following lemma to verify nonsingularity. The idea of using this method was suggested to us by David Holmes.

\begin{lem}\label{nonsing} Let $X$ be a nonsingular projective model for a geometrically irreducible curve $Y$ over $\Q$ and let $p$ be a prime. Denote by $\tilde{X}$ the reduction of $X$ mod $p$ and suppose that $\tilde{X}$ is an integral (i.e. reduced and irreducible) curve. If $\tilde{X}$ has a nonsingular $\mathbb{F}_p$-point and the (geometric) genus of $\tilde{X}$ equals the genus of $Y$, then $X$ has good reduction at $p$.
\end{lem}

\begin{proof} Since $\tilde{X}$ is integral and has a nonsingular $\mathbb{F}_p$ point, it is geometrically integral. The arithmetic genus of $\tilde{X}$ matches the arithmetic genus of $X$, which is the genus of $Y$ (since $X$ is a nonsingular model for $Y$). It follows that the arithmetic genus and geometric genus of $\tilde{X}$ are equal. We now apply \cite[Lemma 0CE4]{stacks-project} to the normalisation of $\tilde{X}$ to conclude that $\tilde{X}$ is nonsingular.  
\end{proof}

We note that verifying that the (geometric) genus of $\tilde{X}$ equals the genus of $X_0(N)$ is a fast computation, since \texttt{Magma} works with the function field of the curve to do this and there is a formula for the genus of $X_0(N)$.

\section{Computing quadratic points}\label{comp_section}

In this section we start by providing an overview of the known results on quadratic points on $X_0(N)$. We then introduce the three methods we use to study quadratic points on $X_0(N)$, and apply them to prove Theorem~\ref{maintm}. 

\subsection{Overview of previously studied $X_0(N)$ and methods}\label{sec:overview}

We start by providing an overview of the known results on quadratic points on $X_0(N)$. We say a point $Q$ is a \emph{quadratic point} on $X_0(N)$ if $Q \in X_0(N)(K) \setminus X_0(N)(\Q)$ for a quadratic field $K$. We will usually consider a quadratic point together with its Galois conjugate, $Q^\sigma$, where $\sigma$ denotes the non-trivial element of $\mathrm{Gal}(K/ \Q)$. A pair of quadratic points gives rise to a rational point on the symmetric square of $X_0(N)$, which we write as an effective degree $2$ divisor $Q + Q^\sigma \in X_0(N)^{(2)}(\Q)$. 

We recall that a smooth projective curve $X/\Q$ of genus $\geq 2$ has infinitely many quadratic points (as we range over all quadratic fields) if and only if it is hyperelliptic or if it is bielliptic with a degree 2 morphism $X \to E$ where $E/\Q$ is an elliptic curve of positive rank over $\Q$  ~\cite{HarrisSilverman}. In these cases, when we say that the quadratic points have been classified, we mean that any quadratic points not arising as part of an infinite (geometric) family have been computed. 

So far, thanks to the work of numerous people across many papers,
the quadratic points have been classified on many modular curves $X_0(N)$ with genus $g(X_0(N)) \geq 2$. We list the results in chronological order:

\begin{enumerate}[label=(\roman*)]
    \item The hyperelliptic $X_0(N)$ with $\rk J_0(N)(\Q)=0$, see \cite{bruin-najman}: This occurs for \[ N \in \{22,23,26,28,29,30,31,33,35,39, 40,41, 46, 47,48,50,59,71 \}.\]
    \item The non-hyperelliptic $X_0(N)$ with $g(X_0(N))\leq 5$ and $\rk J_0(N)(\Q)=0$, see \cite{ozman-siksek}. This occurs for
    \[N \in \{34, 38, 42, 44, 45, 51, 52, 54, 55, 56, 63, 64, 72, 75, 81\}.\]
    \item The $X_0(N)$ with $g(X_0(N))\leq 5$ and $\rk J_0(N)(\Q)>0$, see \cite{box21}. This occurs for
    \[N \in \{37, 43, 53, 61, 57, 65, 67, 73\}.\]
    \item The bielliptic $X_0(N)$ which have not been already dealt with in (i)--(iii), see \cite{najman-vukorepa}. These values are
    \[N \in \{60,62,69,79,83,89,92,94,95,101,119,131\}.\]
    \item Some $X_0(N)$ that were interesting for other reasons: the case $N=77$ was solved in \cite{balcik}, $N=91$ in \cite{Vux91} and the cases $N=125$ and $169$ in \cite{BPN}.
\end{enumerate}

Two broad methods are used to obtain these results. The first is by applying some kind of Mordell--Weil sieve (with different variations according to the properties of the curve $X_0(N)$), and the second is the `going down' method mentioned in the introduction. In order to prove Theorem \ref{maintm} we will use three distinct methods, two of which (namely `rank $0$' and the `Atkin--Lehner sieve') make use of a Mordell--Weil sieve. We split the set $\mathcal{L}$ (defined in (\ref{setL})) into three distinct sets, presented in Table \ref{tab:overview}, according to which method we use. We note that more than one of the methods we use could, in principle, be applied for certain levels $N \in \mathcal{L}_1 \cup \mathcal{L}_2$. We also note that our methods can be used to compute the set of quadratic points for many more levels than we have presented here, although we felt like considering 15 levels was a suitable number to display our techniques. The interested reader may consult the accompanying \texttt{Magma} files where we also ran our code for several other levels up to genus $12$. 

\begin{table}[ht]
    \centering
    \begin{tabular}{ll}
        \toprule
        Method & levels $N$ \\ 
        \midrule   
        Going down & $\mathcal{L}_1 \colonequals \{58, 68, 76 \}$ \\
        Rank $0$ & $\mathcal{L}_2 \colonequals \{ 80, 98, 100 \}$ \\
        Atkin--Lehner sieve   & $\mathcal{L}_3 \colonequals \{ 74, 85, 97, 103, 107, 109, 113, 121, 127 \}$ \\
        \bottomrule
    \end{tabular}   
    \caption{Overview of the methods we use.}
    \label{tab:overview} 
\end{table}
\subsection{Going down} \label{sec:gd}

When the quadratic points on $X_0(M)$ have been already classified, it is often beneficial to use this classification to obtain all the quadratic points on $X_0(N)$ for $N=dM$, where $d > 1$ is an integer. We have two cases: 1) when $X_0(M)$ has finitely many quadratic points, which have all been found, and 2) when $X_0(M)$ has infinitely many quadratic points, and all of them have been described, in the sense that all but finitely many are pullbacks of rational points on some quotient, while the remaining finitely many that are not pullbacks have been explicitly listed. 

The case 1) is of course easier: one need only check whether each of the elliptic curves corresponding to a quadratic point on $X_0(M)$ gives rise to a quadratic point on $X_0(N)$. The cases $M=34$ and $38$, with $d=2$ in each case, fall into this category. 

In case 2) it is necessary for $M$ and $d$ to be coprime as we will use \cite[Proposition 2.2]{najman-vukorepa}. If this is satisfied, then the problem reduces to determining the rational points on several quotients of modular curves. The case $(d,M)=(2,29)$ falls into this category.

We aim to prove the following proposition.

\begin{prop}\label{propL1} Let $N \in \mathcal{L}_1 = \{58,68,76\}$. The finitely many non-cuspidal quadratic points on the curve $X_0(N)$ are displayed in the tables in Section \ref{sec:tables}. Note that all non-cuspidal quadratic points have complex multiplication.
\end{prop}

\begin{proof}
Let $N \in \{58, 68, 76\}$, let $M = N/2$, and let $P \in X_0(N)(K)$ be a quadratic point. The image of $P$ in $X_0(M)(K)$ has the same $j$-invariant as $P$.

For $N = 68$ and $76$, the finitely many quadratic points on $X_0(M)$, together with their $j$-invariants, have been classified in \cite{ozman-siksek}. This provides us with a finite list of possible pairs $(j,L)$ of $j$-invariants and quadratic fields such that $(j(P),K) = (j,L)$. In the case $M = 34$, there are two pairs of non-CM points (denoted $P_5$ and $P_6$ in Table~8.1 of \cite{ozman-siksek}). It can be seen that the corresponding elliptic curves do not admit a cyclic $4$-isogeny, as otherwise this isogeny class (over $L$) would need to contain at least $6$ curves with a $34$-isogeny. For each remaining pair $(j,L)$, we check using a model for $X_0(N)$ and equations for the $j$-map whether there exist any points in $X_0(N)(L)$ with $j$-invariant $j$.

Next, we consider the case $N = 58$. By \cite{bruin-najman}, we know that any elliptic curve with a $29$-isogeny over a quadratic field either corresponds to one of the finitely many points with exceptional $j$-invariants (listed in \cite[Table 5]{bruin-najman}), or is a $\Q$-curve  of degree $29$ which in addition has a $2$-isogeny. We first check that none of the elliptic curves with the exceptional $j$-invariants has a point of order $2$, and so the image of $P$ in $X_0(M)(K)$ must correspond to a $\Q$-curve of degree $29$ with a $2$-isogeny. We apply~\cite[Proposition~2.2]{najman-vukorepa} and conclude that either $P$ corresponds to a rational point on either $X_0(58)/w_{29}$ or $X_0^+(116)$, or that $P$ is a CM point.

Applying the classical Chabauty--Coleman method using \texttt{Magma}'s inbuilt function to compute the $\Q$-rational points on the genus $2$ curve $X_0(58)/w_{29}$ (whose Jacobian has rank $1$ over $\Q$), we obtain that it has precisely $8$ rational points. We compute the pullbacks of rational points on $X_0(58)/w_{29}$ and obtain the $4$ rational cusps and $6$ pairs of quadratic points on $X_0(58)$, all of which correspond to CM curves. These are the points $P_1, \dots, P_6$ in Table \ref{table58}.

Next, the non-cuspidal rational points on $X_0^+(116)$ correspond to CM points by~\cite[Theorem~0.1]{Momose1987}, and so $P$ must be a CM point. From the data associated to the paper \cite{CGPS21} (available at \url{https://github.com/fsaia/least-cm-degree/blob/master/Least%20Degrees/X0}), we obtain five possible pairs $(j,L)$ of $j$-invariants $j$ and quadratic fields $L$ for the pair $(j(P), K)$. Each of these pairs in fact already arises as the $j$-invariant and field of definition of one of the points $P_1, \dots, P_6$  found in the previous paragraph. It remains to check whether there are additional points on $X_0(58)$, different than the ones we have already discovered, corresponding to the pairs $(j,L)$ we have already found. It turns out that there are, and we discover two additional points, the points $P_7$ and $P_8$ in Table \ref{table58}, corresponding to $(j,L)=\left(-3375, \Q(\sqrt{-7}) \right)$.
\end{proof}

 We note that for the cases $68$ and $76$, it would also be possible in the proof above to construct the map from $X_0(N)$ to $X_0(M)$ and directly pull back points. 

\subsection{Rank $0$} \label{sec:rank0}

A method for computing the quadratic points on non-hyperelliptic curves $X_0(N)$ whose Jacobian $J_0(N)$ has rank $0$ over $\Q$ is developed in \cite{ozman-siksek}. We call this the `rank $0$' method. In this section, we use this method to prove the following result.

\begin{prop}\label{propL2} Let $N \in \mathcal{L}_2 = \{80, 98, 100\}$. The finitely many non-cuspidal quadratic points on the curve $X_0(N)$ are displayed in the Tables in Section \ref{sec:tables}. Note that all non-cuspidal points have complex multiplication. 
\end{prop}

We note that the curve $X_0(80)$ has two pairs of cuspidal quadratic points defined over the field $\Q(\sqrt{-1})$.

In \cite{ozman-siksek}, the `rank $0$' method is used on curves up to genus $5$. Although we use the same method, the code of \cite{ozman-siksek} could not be used to extend the computations to curves of larger genus in a reasonable time. The reason we are able to work with higher genus curves is thanks to the models of $X_0(N)$ and equations for the $j$-maps that we use (described in Section \ref{models_section}), as well as our method for verifying nonsingularity at a given prime (see Lemma \ref{nonsing}). 

In order to use the `rank $0$' method, we must verify that $J_0(N)(\Q)$ has rank $0$. Using the results of \cite{GrossZagier1986} and \cite{KolyvaginLogachev} (see \cite[p.~4]{ACKP22} for further details), we were able to compute the rank of $J_0(N)(\Q)$ for every level $N \in \mathcal{L}$. The set $\mathcal{L}_2 = \{80, 98, 100 \}$ consists of the levels $N \in \mathcal{L}$ such that the rank of $J_0(N)(\Q)$ is $0$, and for which we could not use the `going down' method (we note that the method is not applicable in the cases $N = 80$ and $100$ due to the coprimality assumption on $M$ and $d$ specified in Section \ref{sec:gd}). 

We briefly outline the method used in \cite{ozman-siksek}, which uses a type of Mordell--Weil sieve to prove that a given list of quadratic points on $X_0(N)$ is complete. Let $P_0 \in X_0(N)(\Q)$ denote a rational cusp and let $\iota\colon X_0(N)^{(2)}(\Q) \hookrightarrow J_0(N)(\Q)$ denote the Abel--Jacobi map with basepoint $2P_0$, which is injective since $X_0(N)$ is non-hyperelliptic. Suppose that $D = Q + Q^\sigma \in X_0(N)^{(2)}(\Q)$ is a hypothetical unknown quadratic point. 

We compute the \textit{rational cuspidal divisor class group} $C_0(N)(\Q)$ of $X_0(N)$, which is defined as the subgroup of $J_0(N)$ generated by the linear equivalence classes of the degree 0 $\Q$-rational cuspidal divisors (divisors supported only on the cusps). By bounding the index of $C_0(N)(\Q)$  in $J_0(N)(\Q)$, we see that $I \cdot J_0(N)(\Q) \subseteq C_0(N)(\Q)$ for a positive integer $I$. For any level $N$, by the Manin--Drinfeld theorem, we have that $C_0(N)(\Q) \subseteq J_0(N)(\Q)_\mathrm{tors}$, and the generalised Ogg conjecture asserts that this inclusion is in fact an equality (see \cite[p.~2463]{ozman-siksek}). We also note that for any prime level $N$, thanks to the work of Mazur \cite[Theorem~1]{mazur:isogenies}, we know that $C_0(N)(\Q) = J_0(N)(\Q)_\mathrm{tors}$ and that $J_0(N)(\Q)_\mathrm{tors}$ is generated by the difference of the two rational cusps.

The point $D$ satisfies $[D-2P_0]=I \cdot [D^\prime]$ for some $[D^\prime] \in J_0(N)(\Q)$. We then employ a Mordell--Weil sieve, hoping to eliminate all possibilities for $D^\prime$, and therefore achieving a contradiction.

In order to prove Proposition \ref{propL2}, we will make use of the following lemma.

\begin{lem}\label{tors_lem0}
Let $C_0(N)(\Q)$ denote the rational cuspidal divisor class group of $X_0(N)$. For $N \in \{98, 100\}$ we have $J_0(N)(\Q) = C_0(N)(\Q)$. In particular, the generalised Ogg conjecture holds for these values of $N$. The structure of these groups is displayed in Tables \ref{table98} and \ref{table100} in Section \ref{sec:tables}.

For $N = 80$, we have that $4 \cdot J_0(N)(\Q) \subseteq C_0(N)(\Q)$. The structure of the group $C_0(N)(\Q)$ and the possibilities for the quotient group $J_0(N)(\Q) / C_0(N)(\Q)$ are displayed in Table \ref{table80} of Section  \ref{sec:tables}.
\end{lem}

\begin{proof}
We first compute the group $C_0(N)(\Q) \subseteq J_0(N)(\Q)$ using the code of \cite{ozman-siksek}. 

In the cases $N = 98$ and $100$ we simply compute an upper bound on $\#J_0(N)(\Q)$ by reducing modulo odd primes $p \nmid N$. This allowed us to prove that $J_0(N)(\Q) = C_0(N)(\Q)$.

In the case $N = 80$ we proceed similarly, but instead compute a (finite) supergroup of $J_0(N)(\Q)$ by considering the group structure of $J_0(N)(\F_p)$ for some odd primes $p$. This allowed us to prove that $J_0(N)(\Q) / C_0(N)(\Q)$ is isomorphic to a subgroup of $(\Z / 2\Z)^2$.
\end{proof}

\begin{proof}[Proof of Proposition \ref{propL2}]
We applied the method described above. Thanks to Lemma \ref{tors_lem0}, we may set $I = 4$ when $N = 80$ and set $I = 1$ for $N \in \{98,100\}$. The method was successful in each case.
\end{proof}

\subsection{Atkin--Lehner sieve} \label{sec:sieve}

In this section we aim to prove the following result.
\begin{prop}\label{propL3} Let $N \in \mathcal{L}_3 = \{74,85,97,103,107,109,113,121,127\}$. The finitely many non-cuspidal quadratic points on the curve $X_0(N)$ are displayed in the tables in Section \ref{sec:tables}.
\end{prop}

We will prove this proposition in two main stages:

\begin{enumerate}[label = (\Roman*)]
    \item Apply a Mordell--Weil sieve with respect to a suitably chosen Atkin--Lehner involution $w_d$ to determine all quadratic points on $X_0(N)$ that either do not arise as pullbacks of rational points on $X_0(N) /w_d$, or do not arise as fixed points of $w_d$.
    \item Compute the rational points on $X_0(N) / w_d$ and any fixed points of $w_d$ (defined over quadratic fields) on $X_0(N)$.
\end{enumerate}

The Mordell--Weil sieve we apply is an adaptation of the sieve employed by Najman and Vukorepa for $X_0(131)$ in \cite[p.~1806]{najman-vukorepa}, which in turn builds on the results of Box \cite{box21} and Siksek \cite{siksek:symchabauty}. For further background on the Mordell--Weil sieve in general, we refer to~\cite{BruinStoll2010}.

Starting with a model for $X_0(N)$, the sieve uses the following additional inputs:
\begin{itemize}
\item An Atkin--Lehner involution $w_d$ such that the ranks of $J_0(N)(\Q)$ and $(1+w_d)(J_0(N))(\Q)$ are equal.
\item Generators of $J_0(N)(\Q)_\mathrm{tors}$.
\item A set of odd primes $\mathcal{P}$ of good reduction for our model of $X_0(N)$.
\item A (possibly empty) set of quadratic points on $X_0(N)$ that do not arise as pullbacks of rational points on $X_0(N)/w_d$.
\end{itemize}

We write $X = X_0(N)$ and work with rational points on the symmetric square $X^{(2)}$  of this curve (as described in Section \ref{sec:overview}). Let $D_\infty$ be a sum of two rational cusps such that $w_d$ acts trivially on $D_\infty$ (e.g. $D_\infty = \infty + w_d(\infty)$). We use the divisor $D_\infty$ as the basepoint of the Abel--Jacobi map $\iota\colon X^{(2)} \hookrightarrow J(X)$, which is injective since $X$ is non-hyperelliptic. Since the ranks of $J(X)(\Q)$ and $(1+w_d)(J(X)(\Q))$ are equal, we see that $(1-w_d)(J(X)(\Q)) \subseteq J(X)(\Q)_\mathrm{tors}$. Let $p$ be a prime of good reduction for $X$. The following commutative diagram describes the set-up of the sieve:
\begin{equation*}\label{SieveDiagram}  \begin{tikzcd}[sep = large]
X^{(2)}(\Q) \arrow[hook]{r}{\iota}  \arrow{d}{\red_p} &  J(X)(\Q)  \arrow{d}{\red_p} \arrow{r}{1-w_d} & J(X)(\Q)_\mathrm{tors} \arrow[hook]{d}{\red_p}
\\ 
X^{(2)}(\mathbb{F}_p) \arrow{r}{\tilde{\iota}} & J(X)(\mathbb{F}_p) \arrow{r}{1-\tilde{w}_d} &  J(X)(\F_p)
\end{tikzcd} \end{equation*}
Here, $\red_p$ denotes reduction mod $p$ (which we note is injective on $J(X)(\Q)_\mathrm{tors}$), and $\tilde{\iota}$ and $\tilde{w}_d$ are the reductions mod $p$ of $\iota$ and $w_d$ respectively. 

Given a hypothetical unknown quadratic point $Q \in X^{(2)}(\Q)$ that does not map to a rational point on $X /w_d$, we first compute a set $S_Q \subseteq X^{(2)}(\mathbb{F}_p)$ of possibilities for $\red_p(Q)$. In order to construct the set $S_Q$, we consider each point $\red_p(R) \in X^{(2)}(\F_p)$ that is the reduction of a known non-pullback point $R \in X^{(2)}(\Q)$ with respect to $w_d$. We attempt to prove that $\red_p(Q) \ne \red_p(R)$ by applying a symmetric Chabauty criterion as stated in \cite[Theorem~2.1]{box21}. By the commutativity of the diagram, we then have that \[ ( (1-w_d) \circ \iota)(Q) \in W_p \colonequals \red_p^{-1}\big(((1-\tilde{w}_d) \circ \tilde{\iota} )(S_Q)\big) \subseteq J(X)(\Q)_\mathrm{tors}.\]
The set $W_p$ is explicitly computable. We aim to find a set of odd primes $\mathcal{P}$ of good reduction for our model such that \begin{equation*}
\bigcap_{p \in \mathcal{P}} W_p = [0] \in J(X)(\Q)_\mathrm{tors}.
\end{equation*}
If this is the case, then \[ ((1-w_d) \circ \iota)(Q) = (1-w_d)[Q-D_\infty] =  [0]. \]  It follows that $[Q-D_\infty]=w_d([Q-D_\infty])$ and hence, since $w_d$ acts trivially on $D_\infty$, we have that $[Q-w_d(Q)]=0$. Since $X_0(N)$ is non-hyperelliptic, it follows that $Q=w_d(Q)$ (as points in $X^{(2)}(\Q)$). Hence, $Q$ either arises from a pair of quadratic points, each of which is a fixed point of $w_d$, or $Q$ is the pullback of a rational point on $X_0(N)/w_d$ with respect to the quotient map $X_0(N) \rightarrow X_0(N)/w_d$. This completes Stage (I) of the method. 

We note that instead of working with $J(X)(\Q)_\mathrm{tors}$, knowledge of generators for any group $G$ satisfying $(1-w_d)(J(X))(\Q) \subseteq G \subseteq J(X)(\Q)_\mathrm{tors}$ would allow for a similar application of the sieve.

\begin{remark}\label{Chab_rem} 
We make some technical remarks about the application of the symmetric Chabauty criterion (\cite[Theorem~2.1]{box21}) in the sieve. \begin{itemize}
\item The diagonalised models we work with, combined with the fact that the ranks over $\Q$ of the Jacobians of the curves $X$ and $X / w_d $ are equal, allows us to easily compute the appropriate annihilating differentials to apply this criterion. 

\item We only apply the symmetric Chabauty criterion to points in $X^{(2)}(\F_p)$ that are the reductions of known non-pullback points. Although we could also apply it to the reductions of known pullback points, the criterion is guaranteed to fail for such points given our choice of annihilating differentials (which we compute as in \cite[Proposition 3.5]{box21}). However, even if we were able to compute additional annihilating differentials and prove that $\red_p(Q) \ne \red_p(R)$ for each known pullback point $R \in X^{(2)}(\Q)$, the set $W_p$ will still contain the element $[0]$ unless \emph{every} point in $X^{(2)}(\F_p)$ on which $w_d$ acts trivially happens to be the reduction of a known pullback point. This is highly unlikely to occur.

\item We do not make use of the \emph{relative} symmetric Chabauty criterion (\cite[Theorem~2.4]{box21}) in the sieve as it is also very unlikely to provide any additional information (for the same reason as above).
\end{itemize}
\end{remark}

\begin{remark}\label{advantages_rem} We take the opportunity to discuss the advantages and disadvantages of this sieve compared to the sieves used in \cite[pp.~327--328]{box21} and \cite[pp.~258--260]{MJ_cart}. 

The main advantage of the sieve we use is the fact that it is not necessary to work with a finite index subgroup of the Jacobian. This reduces computation time and simplifies the sieving algorithm. It also allows for a greater choice of primes to be used in the sieve when the rank of the Jacobian is $\geq 3$ (c.f. \cite[pp.~258--260]{MJ_cart}). Another advantage is the fact that the sieve does not explicitly work with the quotient curve $X_0(N) /w_d$.

There are certain disadvantages to the sieve we use. First of all, the sieve relies on the existence of a quotient curve with equal rank. Secondly, we also require knowledge of generators for $J(X)(\Q)_\mathrm{tors}$ (or generators of a group $G$ satisfying $(1-w_d)(J(X))(\Q) \subseteq G \subseteq J(X)(\Q)_\mathrm{tors}$, as mentioned above.) This is perhaps the main barrier to extending our computations for all curves $X_0(N)$ up to, say, genus $12$.
\end{remark}

One of the key inputs into the Mordell--Weil sieve is a list of quadratic points that do not arise as pullbacks of rational points $X_0(N) /w_d$  and are not fixed points of $w_d$ (if such points exist and are not input, then the sieve is guaranteed to fail). In order to find these points, we first search for rational points (up to some height bound) on each Atkin--Lehner quotient, and pull these back to $X_0(N)$. This is straightforward thanks to the diagonalised model for $X_0(N)$ that we work with. For each level we considered, this turned out to be sufficient to find all the necessary points. However, if this were not the case, we note that we can also search for quadratic points by intersecting a model for our curve with hyperplanes, as described in \cite[p.~30]{box21}. We significantly improved upon the running time of this method by noting that quadratic points give at most two degrees of linear independence: i.e.~already the first three coordinates of a quadratic point must satisfy a $\Z$-linear relationship. This allows us to look only at hyperplanes of the form $a_1x_1+a_2x_2+a_3x_3=0$. This is  notably faster than going through all the hyperplanes of the form $a_1x_1+\dots+a_gx_g = 0$ (since our curves have $g \geqslant 6$). This method carries over to search for cubic and other low-degree points.

In order to prove Proposition \ref{propL3}, we will first prove two lemmas. We recall from Section \ref{sec:rank0} that $C_0(N)(\Q)$ denotes the rational cuspidal divisor class group of $X_0(N)$. The following lemma is analogous to the cases $N = 98$ and $100$ of Lemma \ref{tors_lem0} and proved in the same way.

\begin{lem}\label{torsion_lem} Let $C_0(N)(\Q)$ denote the rational cuspidal divisor class group of $X_0(N)$. Then for $N \in \{74,85,121\}$ we have $J_0(N)(\Q)_\mathrm{tors} = C_0(N)(\Q)$. In particular, the generalised Ogg conjecture holds for these values of $N$. The structure of these groups is displayed in the tables in Section \ref{sec:tables}.
\end{lem}

\begin{lem}\label{rat_point_lem}
The curve $X_0(74) /w_{37}$ has precisely $9$ rational points and the curve $X_0(85) / w_{85}$ has precisely $8$ rational points.
\end{lem}

\begin{proof} We start by considering the curve $X_0(85) /w_{85}$. In \cite[p.~107]{hyperQ}, the finitely many rational points on the full Atkin--Lehner quotient curve $X_0(85) / \langle w_{85},w_{17}\rangle$ are determined. Since this curve is a quotient of $X_0(85) /w_{85}$, it is straightforward to then verify that $X_0(85)/w_{85}$ has precisely $8$ rational points.

The curve $X_0(74)/w_{37}$ is non-hyperelliptic and the rank of its Jacobian over $\Q$ is smaller than its genus. We successfully applied the classical Chabauty method to determine all the rational points on this curve. We did this using the \texttt{Magma} implementation of classical Chabauty due to Balakrishnan and Tuitman \cite{BalTui}. Some additional details of this computation are provided in Section~\ref{ex_74}.
\end{proof}

\begin{proof}[Proof of Proposition \ref{propL3}]
We will apply the sieve described above. We set $d = N$, unless $N = 74$ in which case we set $d = 37$.
We first verify in each case that the ranks of the Jacobians over $\Q$ of the curves $X_0(N)$ and $X_0(N)/w_d$ are equal by checking that $(1-w_d)J_0(N)(\Q)$ has rank $0$. Next, we use Lemma \ref{torsion_lem} combined with the fact that $C_0(N)(\Q) = J_0(N)(\Q)_\mathrm{tors}$ when $N$ is prime to obtain generators for $J_0(N)(\Q)_\mathrm{tors}$.

We now apply the sieve. In each we were successful using primes in the range $3 \leq p \leq 11$  of good reduction for the curve.

To complete the proof, we proceed with Stage (II) of the method. We first compute any fixed points of $w_d$ defined over quadratic fields, which is a straightforward computation in \texttt{Magma}. Next, when $N = d$ and $N \ne 85$, the rational points on $X_0^+(N) = X_0(N)/w_d$ have been computed for each $N$  we are interested in across a series of papers \cite{ ACKP22, BBB, BDMTV2}. The remaining cases are $(N,d) = (74,37)$ and $(85,85)$ which are covered by Lemma \ref{rat_point_lem}. We then simply pull back the rational points on $X_0(N)/w_d$ to $X_0(N)$ to obtain the remaining quadratic points on $X_0(N)$.
\end{proof}

Propositions \ref{propL1}, \ref{propL2}, and \ref{propL3} combine to prove Theorem \ref{maintm}. 

\begin{remark}\label{running_time_rem} The running times for each level $N \in \mathcal{L}_3$ are displayed in the accompanying \texttt{Magma} files. The running times ranged from 23 seconds (for $N = 109$) to 11 minutes (for $N  =85$). Our implementation significantly reduces the computation time for some previously computed levels. For example, it took under 4 seconds to recover the classification of quadratic points on $X_0(53)$ obtained in \cite{box21}, whilst the original running time using the code of \cite{box21} is 33 minutes. We note that the running time of our algorithm could possibly be improved further by only using the \texttt{Place} command in \texttt{Magma} when working over finite fields (and not using it over $\Q$ as we do in certain instances). 
\end{remark}

\subsubsection{Example: quadratic points on $X_0(74)$} \label{ex_74}

We provide some details of the computations for $X_0(74)$, which is perhaps the most interesting curve we worked with. We note that despite the fact that the quadratic points on $X_0(37)$ have been classified \cite{box21}, their complicated structure means that we cannot compute the quadratic points on $X_0(74)$ using the `going down' method.

The curve $X_0(74)$ has genus $8$, and the quotient curve $X_0(74)/w_{37}$ is non-hyperelliptic of genus $4$. We start by computing a diagonalised model for $X_0(74)$ together with a map to the quotient curve $X_0(74)/w_{37}$ using the methods described in Section \ref{models_section}. 

We verify that the rank of $(1-w_{37})J_0(74)(\Q)$ is $0$, using the techniques of Section \ref{sec:rank0}, and as it will be needed afterwards, we also verify that the rank of the Jacobian of $X_0(74)/w_{37}$ over $\Q$ is $2$. We then verify the computation in Lemma \ref{torsion_lem} that $C_0(74)(\Q) \isom \Z / 3\Z \times \Z / 171 \Z \isom J_0(74)(\Q)_\mathrm{tors}$. We check the first isomorphism using the \texttt{Magma} code of \cite{ozman-siksek}, adapted to work with our diagonalised model. Next, we find that \[ \#J_0(74)(\F_3) = 3^3 \cdot 7^2 \cdot 19 \quad \text{and} \quad \#J_0(74)(\F_5) = 2^8 \cdot 3^3  \cdot 13 \cdot 19. \] This proves that $J_0(74)(\Q)_\mathrm{tors} \isom \Z / 3\Z \times \Z / 171 \Z$. 

Next, we search for quadratic points on $X_0(74)$ by pulling back rational points on the three quotient curves $X_0(74) / w_m$ for $m \in \{2,37,74\}$. We obtained two pairs of quadratic points (the points $P_1$ and $P_2$ in Table \ref{table74} together with their Galois conjugates) that do not arise as pullbacks of rational points on $X_0(74)/w_{37}$. Then, applying the sieve, we found that $\#W_3 = 13$, and that $W_3 \cap W_5 = [0]$, meaning that the sieving process was successful using the primes $3$ and $5$. 

Next, we continue with Stage (II) of the method. We compute the fixed points of $w_{37}$ on $X_0(74)$ and find that there is a single quadratic point in this fixed locus, the point $P_{11}$ in Table \ref{table74}. Finally, it remains to verify the computation in Lemma \ref{rat_point_lem} that the curve $X_0(74)/w_{37}$ has precisely $9$ rational points. We do this using the code of Balakrishnan and Tuitman \cite{BalTui} on a plane model of this curve found by a function from \cite{AABCCKW}. We use the  \texttt{effective\textunderscore chabauty} function from \cite{BalTui} with $p=11$ to show that $\Q$-rational points on $X_0(74)/w_{37}$ consist of exactly $9$ points. In order to use this function, we first check that the differences of the $9$ low-height rational points on this
quotient curve generate a finite index subgroup of its
Jacobian (which we recall has rank $2$). By considering the differences of pairs of rational points, and by working modulo $3$ and $5$, we were able to find two independent rational points of infinite order on the Jacobian.

\section{Tables}\label{sec:tables}

In the following tables we have included the non-cuspidal quadratic point data (up to Galois conjugation) for the curves $X_0(N)$ for $N \in \mathcal{L}$. For each quadratic point we have displayed its field of definition, it's $j$-invariant, and the corresponding CM discriminant when applicable. We have then displayed the action of the Atkin--Lehner involutions on each point. In the cases for which we computed the structure of $J_0(N)(\Q)$, this information is also included. 

In addition to the data presented in these tables (which is independent of our chosen models), for each curve $X_0(N)$, the projective models we used, equations for the Atkin--Lehner involutions, and the coordinates of each quadratic point are available in the accompanying \texttt{Magma} files. 


\begin{table}[!ht]
    \caption{All non-cuspidal quadratic points on \boldmath $X_0(58)$}
    \label{table58}
    \begin{flalign*}
	& \text{Genus: } 6 \\
    \end{flalign*} 
    {\small
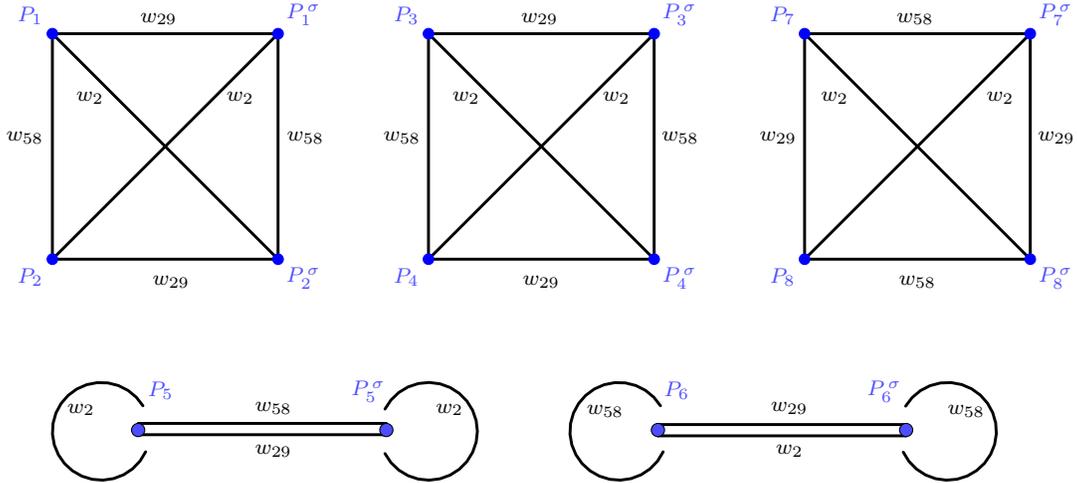
    \begin{tabular}{cccc}
        \toprule
        Point & Field  & $j$-invariant & CM \\ [0.3ex]
        \midrule  
        $P_1$ & $\Q(\sqrt{-1})$ & $1728$ &  $-4$\\ [0.5ex]	
        $P_2$ & $\Q(\sqrt{-1})$ & $287496$ & $-16$  \\[0.5ex]  
        $P_3$ & $\Q(\sqrt{-7})$  & $-3375$ & $-7$ \\[0.5ex]
        $P_4$ & $\Q(\sqrt{-7})$ & $16581375$  &  $-28$  \\[0.5ex]
        $P_5$ & $\Q(\sqrt{-1})$ & $1728$ & $-4$ \\ [0.5ex]
        $P_6$ & $\Q(\sqrt{29})$ & $-56147767009798464000 \sqrt{29} + 302364978924945672000$ & $-232$  \\ [0.5ex] 
        $P_7$ & $\Q(\sqrt{-7})$  & $-3375$ & $-7$ \\[0.5ex]
        $P_8$ & $\Q(\sqrt{-7})$  & $-3375$ & $-7$ \\[0.5ex]
         \bottomrule
    \end{tabular}
    } \\
\begin{tikzpicture}[line cap=round,line join=round,>=triangle 45,x=1.0cm,y=1.0cm]
\tikzmath{\x1 = 0.35; \y1 =-0.05; \z1=180; \w1=0.2;} 
\clip(-1.7,-0.2596390243902512) rectangle (14.8,7.271873170731704);
\draw [line width=1.1pt] (0.,6.)-- (3.,6.);
\draw [line width=1.1pt] (0.,3.)-- (3.,3.);
\draw [line width=1.1pt] (0.,6.)-- (0.,3.);
\draw [line width=1.1pt] (3.,6.)-- (3.,3.);
\draw [line width=1.1pt] (0.,6.)-- (3.,3.);
\draw [line width=1.1pt] (0.,3.)-- (3.,6.);
\draw [line width=1.1pt] (5.,6.)-- (8.,6.);
\draw [line width=1.1pt] (5.,6.)-- (5.,3.);
\draw [line width=1.1pt] (5.,3.)-- (8.,3.);
\draw [line width=1.1pt] (8.,6.)-- (8.,3.);
\draw [line width=1.1pt] (5.,6.)-- (8.,3.);
\draw [line width=1.1pt] (5.,3.)-- (8.,6.);
\draw [line width=1.1pt] (1.14,0.665)-- (4.44,0.665);
\draw [line width=1.1pt] (1.14,0.815)-- (4.44,0.815);
\draw [shift={(0.65,0.71)},line width=1.1pt]  plot[domain=0.55:5.85,variable=\t]({1.*0.65*cos(\t r)+0.*0.65*sin(\t r)},{0.*0.65*cos(\t r)+1.*0.65*sin(\t r)});
\draw [shift={(5.,0.71)},line width=1.1pt]  plot[domain=-2.6685113621829863:2.6685113621829863,variable=\t]({1.*0.65*cos(\t r)+0.*0.65*sin(\t r)},{0.*0.65*cos(\t r)+1.*0.65*sin(\t r)});

\draw [line width=1.1pt] (8.05,0.8)-- (11.35,0.8);
\draw [line width=1.1pt] (8.05,0.65)-- (11.35,0.65);
\draw [shift={(7.53,0.71)},line width=1.1pt]  plot[domain=0.55:5.85,variable=\t]({1.*0.65*cos(\t r)+0.*0.65*sin(\t r)},{0.*0.65*cos(\t r)+1.*0.65*sin(\t r)});
\draw [shift={(11.9,0.71)},line width=1.1pt]  plot[domain=-2.7:2.6,variable=\t]({1.*0.65*cos(\t r)+0.*0.65*sin(\t r)},{0.*0.65*cos(\t r)+1.*0.65*sin(\t r)});

\draw [line width=1.1pt] (10.,6.)-- (13.,6.);
\draw [line width=1.1pt] (10.,6.)-- (10.,3.);
\draw [line width=1.1pt] (10.,3.)-- (13.,3.);
\draw [line width=1.1pt] (10.,6.)-- (13.,3.);
\draw [line width=1.1pt] (10.,3.)-- (13.,6.);
\draw [line width=1.1pt] (13.,6.)-- (13.,3.);
\begin{scriptsize}
\node [label={[label distance=\dist mm]\NE: {\textcolor{ududff}{$P_1$}}},circle,fill=blue,draw=blue,scale=0.5](A1) at (0,6) {};
\node [label={[label distance=\distlarger mm]\NW: {\textcolor{ududff}
{$P_1^{\, \sigma}$}}},circle,fill=blue,draw=blue,scale=0.5](A1) at (3,6) {};
\node [label={[label distance=\dist mm]\SE: {\textcolor{ududff}{$P_2$}}},circle,fill=blue,draw=blue,scale=0.5](A1) at (0,3) {};
\node [label={[label distance=\distlarger mm]\SW: {\textcolor{ududff}
{$P_2^{\, \sigma}$}}},circle,fill=blue,draw=blue,scale=0.5](A1) at (3,3) {};
\draw[color=black] (1.415,6.2) node {$w_{29}$};
\draw[color=black] (1.58,2.7) node {$w_{29}$};
\draw[color=black] (-0.365,4.62) node {$w_{58}$};
\draw[color=black] (3.365,4.62) node {$w_{58}$};
\draw[color=black] (0.5,5.15) node {$w_{2}$};
\draw[color=black] (2.5,5.15) node {$w_{2}$};
\node [label={[label distance=\dist mm]\SE: {\textcolor{ududff}{$P_4$}}},circle,fill=blue,draw=blue,scale=0.5](A1) at (5,3) {};
\node [label={[label distance=\distlarger mm]\SW: {\textcolor{ududff}
{$P_4^{\, \sigma}$}}},circle,fill=blue,draw=blue,scale=0.5](A1) at (8,3) {};
\node [label={[label distance=\dist mm]\NE: {\textcolor{ududff}{$P_3$}}},circle,fill=blue,draw=blue,scale=0.5](A1) at (5,6) {};
\node [label={[label distance=\distlarger mm]\NW: {\textcolor{ududff}
{$P_3^{\, \sigma}$}}},circle,fill=blue,draw=blue,scale=0.5](A1) at (8,6) {};
\draw[color=black] (6.48,6.2) node {$w_{29}$};
\draw[color=black] (4.65,4.62) node {$w_{58}$};
\draw[color=black] (6.5,2.7) node {$w_{29}$};
\draw[color=black] (8.35,4.62) node {$w_{58}$};
\draw[color=black] (5.5,5.15) node {$w_{2}$};
\draw[color=black] (7.5,5.15) node {$w_{2}$};
\draw [fill=ududff] (1.14,0.7272487804877982) circle (2.5pt);
\draw[color=ududff] (1.44,1.25) node {$P_5$};
\draw [fill=ududff] (4.44,0.7272487804877982) circle (2.5pt);
\draw[color=ududff] (4.2,1.24) node {$P_5^{\, \sigma}$};
\draw[color=black] (2.94,1.02) node {$w_{58}$};
\draw[color=black] (2.94,0.45) node {$w_{29}$};
\draw[color=black] (0.38,1.) node {$w_2$};
\draw[color=black] (5.28,1.) node {$w_2$};
\draw [fill=ududff] (8.05,0.7272487804877982) circle (2.5pt);
\draw[color=ududff] (8.3,1.25) node {$P_{6}$};
\draw [fill=ududff] (11.35,0.7272487804877982) circle (2.5pt);
\draw[color=ududff] (11.06,1.24) node {$P_{6}^{\, \sigma}$};
\draw[color=black] (9.8,0.45) node {$w_2$};
\draw[color=black] (9.8,1.02) node {$w_{29}$};
\draw[color=black] (7.35,0.98) node {$w_{58}$};
\draw[color=black] (12.15,0.98) node {$w_{58}$};
\node [label={[label distance=\dist mm]\SE: {\textcolor{ududff}{$P_8$}}},circle,fill=blue,draw=blue,scale=0.5](A1) at (10,3) {};
\node [label={[label distance=\distlarger mm]\SW: {\textcolor{ududff}
{$P_8^{\, \sigma}$}}},circle,fill=blue,draw=blue,scale=0.5](A1) at (13,3) {};
\node [label={[label distance=\dist mm]\NE: {\textcolor{ududff}{$P_7$}}},circle,fill=blue,draw=blue,scale=0.5](A1) at (10,6) {};
\node [label={[label distance=\distlarger mm]\NW: {\textcolor{ududff}
{$P_7^{\, \sigma}$}}},circle,fill=blue,draw=blue,scale=0.5](A1) at (13,6) {};
\draw[color=black] (11.47,6.2) node {$w_{58}$};
\draw[color=black] (9.65,4.62) node {$w_{29}$};
\draw[color=black] (11.5,2.7) node {$w_{58}$};
\draw[color=black] (10.4,5.15) node {$w_2$};
\draw[color=black] (12.6,5.15) node {$w_2$};
\draw[color=black] (13.35,4.62) node {$w_{29}$};
\end{scriptsize}
\end{tikzpicture}
\end{table}


\begin{table}[!ht]\label{68}
    \caption{All non-cuspidal quadratic points on \boldmath $X_0(68)$}
    \label{table68}
    \begin{flalign*}
	& \text{Genus: } 7 \\
    \end{flalign*} 
    {\small
    \begin{tabular}{cccc}
	\toprule
	Point & Field  & $j$-invariant & CM \\ [0.3ex]
	\midrule 
	$P_1$ & $\Q(\sqrt{-1})$ & $1728$ &  $-4$\\ [0.5ex]
	$P_2$ & $\Q(\sqrt{-1})$ & $287496$ & $-16$ \\ [0.5ex]
	$P_3$ & $\Q(\sqrt{-1})$  & $287496$ & $-16$ \\ [0.5ex] 
	\bottomrule
    \end{tabular}
    } \\ 
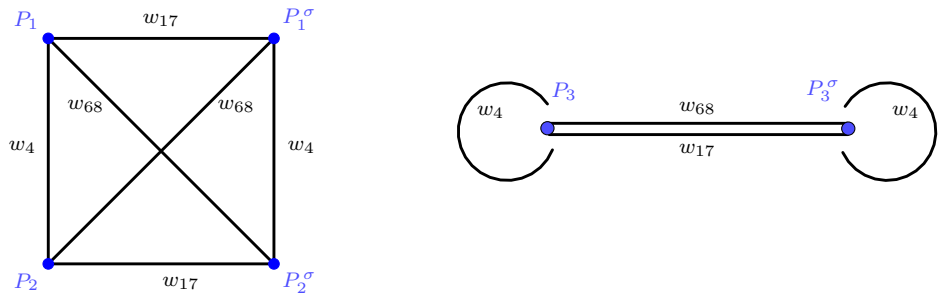
\vspace{10pt}
\begin{tikzpicture}[line cap=round,line join=round,>=triangle 45,x=1.0cm,y=1.0cm]
\clip(-6.084311271347847,2.315162519980797) rectangle (6.873260078590276,7.039910758690617);
\draw [line width=1.1pt] (-5.,6.)-- (-5.,3.);
\draw [line width=1.1pt] (-5.,6.)-- (-2.,6.);
\draw [line width=1.1pt] (-2.,6.)-- (-2.,3.);
\draw [line width=1.1pt] (-5.,3.)-- (-2.,3.);
\draw [line width=1.1pt] (-5.,6.)-- (-2.,3.);
\draw [line width=1.1pt] (-5.,3.)-- (-2.,6.);
\draw [line width=1.1pt] (1.66,4.87) -- (5.63,4.87);
\draw [line width=1.1pt] (1.66,4.72) -- (5.63,4.72);
\draw [shift={(1.1,4.76)},line width=1.1pt]  plot[domain=0.6:5.9,variable=\t]({1.*0.65*cos(\t r)+0.*0.65*sin(\t r)},{0.*0.65*cos(\t r)+1.*0.65*sin(\t r)});
\draw [shift={(6.15,4.76)},line width=1.1pt]  plot[domain=-2.7:2.6,variable=\t]({1.*0.65*cos(\t r)+0.*0.65*sin(\t r)},{0.*0.65*cos(\t r)+1.*0.65*sin(\t r)});
\begin{scriptsize}
\node [label={[label distance=\dist mm]\NE: {\textcolor{ududff}{$P_1$}}},circle,fill=blue,draw=blue,scale=0.5](A1) at (-5,6) {};
\node [label={[label distance=\distlarger mm]\NW: {\textcolor{ududff}
{$P_1^{\, \sigma}$}}},circle,fill=blue,draw=blue,scale=0.5](A1) at (-2,6) {};
\node [label={[label distance=\dist mm]\SE: {\textcolor{ududff}{$P_2$}}},circle,fill=blue,draw=blue,scale=0.5](A1) at (-5,3) {};
\node [label={[label distance=\distlarger mm]\SW: {\textcolor{ududff}
{$P_2^{\, \sigma}$}}},circle,fill=blue,draw=blue,scale=0.5](A1) at (-2,3) {};
\draw[color=black] (-5.344168306914762,4.55) node {$w_{4}$};
\draw[color=black] (-3.5,6.25) node {$w_{17}$};
\draw[color=black] (-1.6333145400310738,4.55) node {$w_{4}$};
\draw[color=black] (-3.2352678055163815,2.75) node {$w_{17}$};
\draw[color=black] (-4.5, 5.1) node {$w_{68}$};
\draw[color=black] (-2.5, 5.1) node {$w_{68}$};
\draw [fill=ududff] (1.6358458536585367,4.804653658536576) circle (2.5pt);
\draw[color=ududff] (1.86,5.3) node {$P_3$};
\draw [fill=ududff] (5.63534,4.8) circle (2.5pt);
\draw[color=ududff] (5.3,5.3) node {$P_3^{\, \sigma}$};
\draw[color=black] (3.63,5.07) node {$w_{68}$};
\draw[color=black] (3.63,4.52) node {$w_{17}$};
\draw[color=black] (0.88,5.02) node {$w_4$};
\draw[color=black] (6.4,5.02) node {$w_4$};
\end{scriptsize}
\end{tikzpicture}
\end{table}


\begin{table}[!ht]
    \caption{All non-cuspidal quadratic points on \boldmath $X_0(74)$}
    \label{table74}
    \begin{flalign*}
	& \text{Genus: } 8 \\
        &J_0(74)(\Q) \isom \Z^2\times \Z/3\Z \times \Z/171\Z \\ 
    \end{flalign*}
    {\small
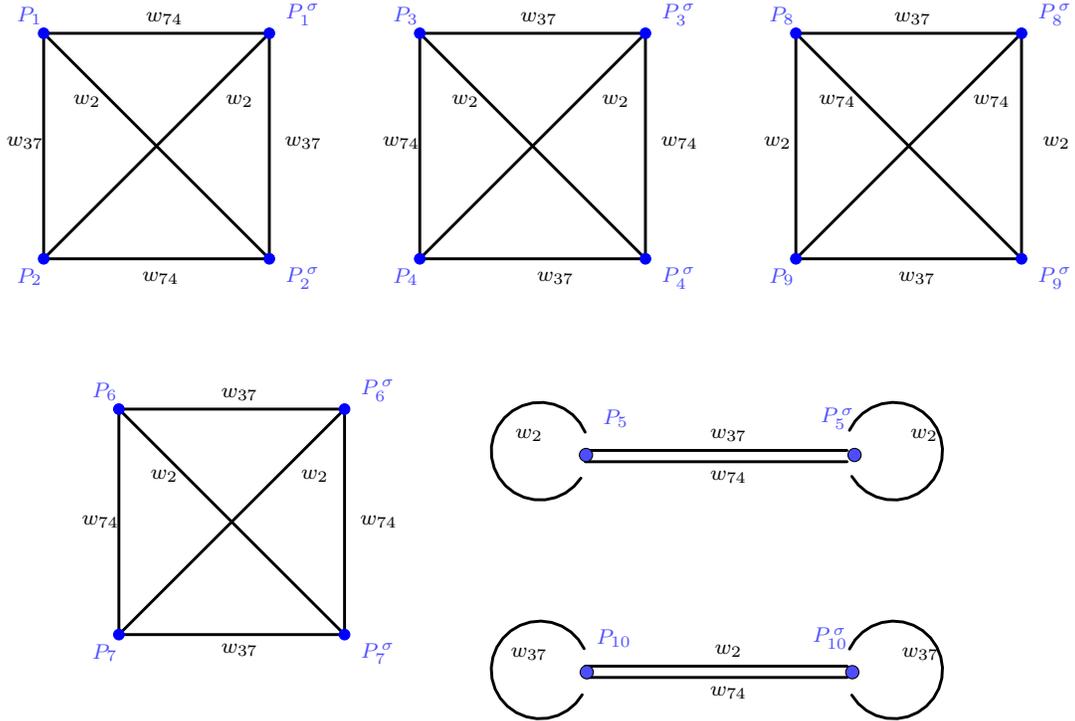
    \begin{tabular}{cccc}
	\toprule
	Point & Field & $j$-invariant & CM \\ [0.3ex]
	\midrule 
	$P_1$ & $\Q(\sqrt{-7})$ & $-3375$ &  $-7$\\ [0.5ex]
        $P_2$ & $\Q(\sqrt{-7})$ & $-3375$ & $-7$\\ [0.5ex]
        $P_3$  & $\Q(\sqrt{-7})$ & $-3375$ & $-7$ \\ [0.5ex]
        $P_4$ & $\Q(\sqrt{-7})$  &  $16581375$ & $-28$ \\  [0.5ex]
        $P_5$ & $\Q(\sqrt{-1})$ &  $1728$ & $-4$ \\ [0.5ex]
        $P_6$ & $\Q(\sqrt{-1})$ &  $1728$ & $-4$ \\ [0.5ex]
        $P_7$ & $\Q(\sqrt{-1})$ & $287496$ & $-16$ \\ [0.5ex]
        $P_8$ & $\Q(\sqrt{-3})$ & $54000$ & $-12$ \\ [0.5ex]
        $P_{9}$ & $\Q(\sqrt{-3})$ & $0$ & $-3$ \\  [0.5ex]
        $P_{10}$ & $\Q(\sqrt{37})$  & $-3260047059360000\sqrt{37} + 19830091900536000$ & $-148$ \\ [0.5ex] 
	\bottomrule
    \end{tabular}
    }\\
    \vspace{-50pt}
\begin{tikzpicture}[line cap=round,line join=round,>=triangle 45,x=1.0cm,y=1.0cm]
\hspace{3pt}
\clip(-5.999549268292689,-4.856458536585376) rectangle (9.08944585365854,9.349531707317073);
\draw [line width=1.1pt] (-5.,6.)-- (-5.,3.);
\draw [line width=1.1pt] (-5.,6.)-- (-2.,6.);
\draw [line width=1.1pt] (-2.,6.)-- (-2.,3.);
\draw [line width=1.1pt] (-5.,3.)-- (-2.,3.);
\draw [line width=1.1pt] (-5.,6.)-- (-2.,3.);
\draw [line width=1.1pt] (-5.,3.)-- (-2.,6.);
\draw [line width=1.1pt] (0.,6.)-- (3.,6.);
\draw [line width=1.1pt] (0.,3.)-- (3.,3.);
\draw [line width=1.1pt] (0.,6.)-- (0.,3.);
\draw [line width=1.1pt] (3.,6.)-- (3.,3.);
\draw [line width=1.1pt] (0.,6.)-- (3.,3.);
\draw [line width=1.1pt] (0.,3.)-- (3.,6.);
\draw [line width=1.1pt] (5.,6.)-- (8.,6.);
\draw [line width=1.1pt] (5.,6.)-- (5.,3.);
\draw [line width=1.1pt] (5.,3.)-- (8.,3.);
\draw [line width=1.1pt] (8.,6.)-- (8.,3.);
\draw [line width=1.1pt] (5.,6.)-- (8.,3.);
\draw [line width=1.1pt] (5.,3.)-- (8.,6.);
\draw [line width=1.1pt] (-4.,1.)-- (-1.,1.);
\draw [line width=1.1pt] (-4.,-2.)-- (-1.,-2.);
\draw [line width=1.1pt] (-4.,1.)-- (-4.,-2.);
\draw [line width=1.1pt] (-1.,1.)-- (-1.,-2.);
\draw [line width=1.1pt] (-4.,1.)-- (-1.,-2.);
\draw [line width=1.1pt] (-4.,-2.)-- (-1.,1.);
\draw [line width=1.1pt] (2.2,0.45)-- (5.68,0.45);
\draw [line width=1.1pt] (2.2,0.3)-- (5.68,0.3);
\draw [shift={(1.6,0.44)},line width=1.1pt]  plot[domain=0.4:5.7,variable=\t]({1.*0.65*cos(\t r)+0.*0.65*sin(\t r)},{0.*0.65*cos(\t r)+1.*0.65*sin(\t r)});
\draw [shift={(6.3,0.44)},line width=1.1pt]  plot[domain=-2.6:2.7,variable=\t]({1.*0.65*cos(\t r)+0.*0.65*sin(\t r)},{0.*0.65*cos(\t r)+1.*0.65*sin(\t r)});
\draw [line width=1.1pt] (2.2,-2.42)-- (5.68,-2.42);
\draw [line width=1.1pt] (2.2,-2.57)-- (5.68,-2.57);
\draw [shift={(1.6,-2.47)},line width=1.1pt]  plot[domain=0.45:5.75,variable=\t]({1.*0.65*cos(\t r)+0.*0.65*sin(\t r)},{0.*0.65*cos(\t r)+1.*0.65*sin(\t r)});
\draw [shift={(6.3,-2.47)},line width=1.1pt]  plot[domain=-2.6:2.7,variable=\t]({1.*0.65*cos(\t r)+0.*0.65*sin(\t r)},{0.*0.65*cos(\t r)+1.*0.65*sin(\t r)});
\begin{scriptsize}
\node [label={[label distance=\dist mm]\NE: {\textcolor{ududff}{$P_1$}}},circle,fill=blue,draw=blue,scale=0.5](A1) at (-5,6) {};
\node [label={[label distance=\distlarger mm]\NW: {\textcolor{ududff}
{$P_1^{\, \sigma}$}}},circle,fill=blue,draw=blue,scale=0.5](A1) at (-2,6) {};
\node [label={[label distance=\dist mm]\SE: {\textcolor{ududff}{$P_2$}}},circle,fill=blue,draw=blue,scale=0.5](A1) at (-5,3) {};
\node [label={[label distance=\distlarger mm]\SW: {\textcolor{ududff}
{$P_2^{\, \sigma}$}}},circle,fill=blue,draw=blue,scale=0.5](A1) at (-2,3) {};
\draw[color=black] (-5.35,4.55) node {$w_{37}$};
\draw[color=black] (-3.5,6.2) node {$w_{74}$};
\draw[color=black] (-1.65,4.55) node {$w_{37}$};
\draw[color=black] (-3.55,2.75) node {$w_{74}$};
\draw[color=black] (-4.53,5.1) node {$w_2$};
\draw[color=black] (-2.5,5.1) node {$w_2$};
\node [label={[label distance=\dist mm]\NE: {\textcolor{ududff}{$P_3$}}},circle,fill=blue,draw=blue,scale=0.5](A1) at (0,6) {};
\node [label={[label distance=\distlarger mm]\NW: {\textcolor{ududff}
{$P_3^{\, \sigma}$}}},circle,fill=blue,draw=blue,scale=0.5](A1) at (3,6) {};
\node [label={[label distance=\dist mm]\SE: {\textcolor{ududff}{$P_4$}}},circle,fill=blue,draw=blue,scale=0.5](A1) at (0,3) {};
\node [label={[label distance=\distlarger mm]\SW: {\textcolor{ududff}
{$P_4^{\, \sigma}$}}},circle,fill=blue,draw=blue,scale=0.5](A1) at (3,3) {};
\draw[color=black] (1.48,6.2) node {$w_{37}$};
\draw[color=black] (1.7007726829268284,2.75) node {$w_{37}$};
\draw[color=black] (-0.35,4.55) node {$w_{74}$};
\draw[color=black] (3.35,4.55) node {$w_{74}$};
\draw[color=black] (0.5,5.1) node {$w_{2}$};
\draw[color=black] (2.5,5.1) node {$w_{2}$};
\node [label={[label distance=\dist mm]\NE: {\textcolor{ududff}{$P_8$}}},circle,fill=blue,draw=blue,scale=0.5](A1) at (5,6) {};
\node [label={[label distance=\distlarger mm]\NW: {\textcolor{ududff}
{$P_8^{\, \sigma}$}}},circle,fill=blue,draw=blue,scale=0.5](A1) at (8,6) {};
\node [label={[label distance=\dist mm]\SE: {\textcolor{ududff}{$P_9$}}},circle,fill=blue,draw=blue,scale=0.5](A1) at (5,3) {};
\node [label={[label distance=\distlarger mm]\SW: {\textcolor{ududff}
{$P_9^{\, \sigma}$}}},circle,fill=blue,draw=blue,scale=0.5](A1) at (8,3) {};
\draw[color=black] (6.45,6.2) node {$w_{37}$};
\draw[color=black] (4.65,4.55) node {$w_{2}$};
\draw[color=black] (6.50535804878049,2.75) node {$w_{37}$};
\draw[color=black] (8.358,4.55) node {$w_{2}$};
\draw[color=black] (5.45,5.1) node {$w_{74}$};
\draw[color=black] (7.5,5.1) node {$w_{74}$};
\node [label={[label distance=\dist mm]\NE: {\textcolor{ududff}{$P_6$}}},circle,fill=blue,draw=blue,scale=0.5](A1) at (-4,1) {};
\node [label={[label distance=\distlarger mm]\NW: {\textcolor{ududff}
{$P_6^{\, \sigma}$}}},circle,fill=blue,draw=blue,scale=0.5](A1) at (-1,1) {};
\node [label={[label distance=\dist mm]\SE: {\textcolor{ududff}{$P_7$}}},circle,fill=blue,draw=blue,scale=0.5](A1) at (-4,-2) {};
\node [label={[label distance=\distlarger mm]\SW: {\textcolor{ududff}
{$P_7^{\, \sigma}$}}},circle,fill=blue,draw=blue,scale=0.5](A1) at (-1,-2) {};
\draw[color=black] (-2.5,1.2) node {$w_{37}$};
\draw[color=black] (-2.5,-2.2) node {$w_{37}$};
\draw[color=black] (-4.35,-0.5) node {$w_{74}$};
\draw[color=black] (-0.65,-0.5) node {$w_{74}$};
\draw[color=black] (-3.5,0.11) node {$w_{2}$};
\draw[color=black] (-1.5,0.11) node {$w_{2}$};
\draw [fill=ududff] (2.21,0.389629268292676) circle (2.5pt);
\draw[color=ududff] (2.5,0.88) node {$P_5$};
\draw [fill=ududff] (5.78,0.389629268292676) circle (2.5pt);
\draw[color=ududff] (5.45,0.88) node {$P_5^{\, \sigma}$};
\draw[color=black] (4.,0.65) node {$w_{37}$};
\draw[color=black] (4.,0.1) node {$w_{74}$};
\draw[color=black] (1.35,0.65) node {$w_2$};
\draw[color=black] (6.6,0.65) node {$w_2$};
\draw [fill=ududff] (2.22,-2.5) circle (2.5pt);
\draw[color=ududff] (2.475,-2.05) node {$P_{10}$};
\draw [fill=ududff] (5.75,-2.5) circle (2.5pt);
\draw[color=ududff] (5.35,-2.05) node {$P_{10}^{\, \sigma}$};
\draw[color=black] (4.,-2.22) node {$w_2$};
\draw[color=black] (4.,-2.77) node {$w_{74}$};
\draw[color=black] (1.35,-2.25) node {$w_{37}$};
\draw[color=black] (6.55,-2.25) node {$w_{37}$};
\end{scriptsize}
\end{tikzpicture}
\end{table}


\begin{table}[!ht]
    \caption{All non-cuspidal quadratic points on \boldmath $X_0(76)$}
    \label{table76}
    \begin{flalign*}
	& \text{Genus: } 8 \\
    \end{flalign*} 
    {\small
    \begin{tabular}{cccc}
	\toprule
	Point & Field & $j$-invariant & CM \\ [0.3ex]
	\midrule
	$P_1$ & $\Q(\sqrt{-3})$ & 54000 &  $-12$\\ [0.5ex]
	$P_2$ & $\Q(\sqrt{-3})$ & 54000 & $-12$ \\ [0.5ex] 
	\bottomrule
    \end{tabular}
    } \\
\begin{tikzpicture}[line cap=round,line join=round,>=triangle 45,x=1.0cm,y=1.0cm]
\clip(-6.85658341463415,1.5842829268292604) rectangle (-0.26001756097561046,7.6354634146341365);
\draw [line width=1.1pt] (-5.,6.)-- (-5.,3.);
\draw [line width=1.1pt] (-5.,6.)-- (-2.,6.);
\draw [line width=1.1pt] (-2.,6.)-- (-2.,3.);
\draw [line width=1.1pt] (-5.,3.)-- (-2.,3.);
\draw [line width=1.1pt] (-5.,6.)-- (-2.,3.);
\draw [line width=1.1pt] (-5.,3.)-- (-2.,6.);
\begin{scriptsize}
\node [label={[label distance=\dist mm]\NE: {\textcolor{ududff}{$P_1$}}},circle,fill=blue,draw=blue,scale=0.5](A1) at (-5,6) {};
\node [label={[label distance=\distlarger mm]\NW: {\textcolor{ududff}
{$P_1^{\, \sigma}$}}},circle,fill=blue,draw=blue,scale=0.5](A1) at (-2,6) {};
\node [label={[label distance=\dist mm]\SE: {\textcolor{ududff}{$P_2$}}},circle,fill=blue,draw=blue,scale=0.5](A1) at (-5,3) {};
\node [label={[label distance=\distlarger mm]\SW: {\textcolor{ududff}
{$P_2^{\, \sigma}$}}},circle,fill=blue,draw=blue,scale=0.5](A1) at (-2,3) {};
\draw[color=black] (-5.35,4.5) node {$w_{4}$};
\draw[color=black] (-3.45,6.2) node {$w_{76}$};
\draw[color=black] (-1.65,4.5) node {$w_4$};
\draw[color=black] (-3.45,2.8) node {$w_{76}$};
\draw[color=black] (-4.45,5.05) node {$w_{19}$};
\draw[color=black] (-2.5,5.05) node {$w_{19}$};
\end{scriptsize} 
\end{tikzpicture}
\end{table}


\begin{table}[!ht]
    \caption{All non-cuspidal quadratic points on \boldmath $X_0(80)$}
    \label{table80}
    \begin{flalign*}
	& \text{Genus: } 7 \\
	& C_0(80)(\Q) \isom \Z/12\Z \times  \Z/24\Z \times  \Z/24\Z \\ 
        &  J_0(80)(\Q)/C_0(80)(\Q) \isom  0, \Z/ 2\Z  \text{ or }  (\Z/ 2\Z)^2 \\
    \end{flalign*} 
    No non-cuspidal quadratic points.
 \end{table}


\begin{table}[!ht]
    \caption{All non-cuspidal quadratic points on \boldmath $X_0(85)$ }
    \label{table85}
    \begin{align*}
	& \text{Genus: } 7 \\
        & J_0(85)(\Q) \isom \Z^2 \times \Z/8\Z \times \Z/48\Z \\
    \end{align*}
    {\small
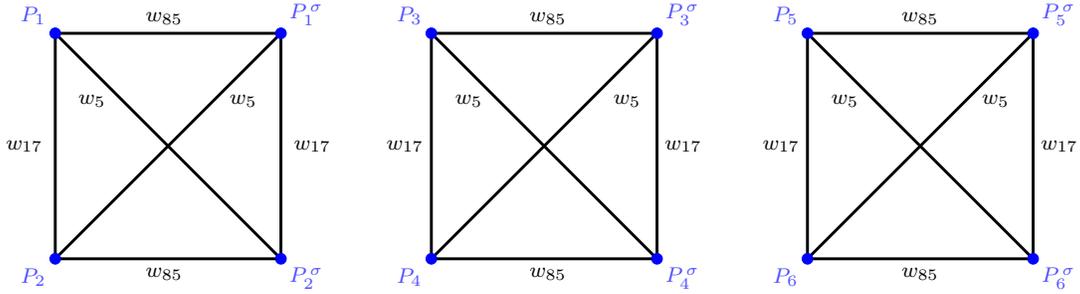
    \begin{tabular}{cccc}
	\toprule
	Point & Field  & $j$-invariant & CM \\ [0.3ex]
	\midrule 
        $P_1$ & $\Q(\sqrt{-19})$ & $-884736$ & $-19$ \\ [0.5ex]
        $P_2$ & $\Q(\sqrt{-19})$  & $-884736$ & $-19$ \\ [0.5ex]
        $P_3$ & $\Q(\sqrt{-1})$  & $1728$ & $-4$ \\ [0.5ex]
        $P_4$ & $\Q(\sqrt{-1})$  & $1728$ & $-4$ \\ [0.5ex]
        $P_5$ & $\Q(\sqrt{-1})$ & $287496$ & $-16$ \\ [0.5ex]
        $P_6$ & $\Q(\sqrt{-1})$  & $287496$ & $-16$  \\ [0.5ex] 
        \bottomrule    
    \end{tabular}
    } \\
    \vspace{-60pt}   
    \hspace{-25pt} 
\begin{tikzpicture}[line cap=round,line join=round,>=triangle 45,x=1.0cm,y=1.0cm]
\clip(-6.3,1.) rectangle (8.6,9.6);
\draw [line width=1.1pt] (-5.,6.)-- (-5.,3.);
\draw [line width=1.1pt] (-5.,6.)-- (-2.,6.);
\draw [line width=1.1pt] (-2.,6.)-- (-2.,3.);
\draw [line width=1.1pt] (-5.,3.)-- (-2.,3.);
\draw [line width=1.1pt] (-5.,6.)-- (-2.,3.);
\draw [line width=1.1pt] (-5.,3.)-- (-2.,6.);
\draw [line width=1.1pt] (0.,6.)-- (3.,6.);
\draw [line width=1.1pt] (0.,3.)-- (3.,3.);
\draw [line width=1.1pt] (0.,6.)-- (0.,3.);
\draw [line width=1.1pt] (3.,6.)-- (3.,3.);
\draw [line width=1.1pt] (0.,6.)-- (3.,3.);
\draw [line width=1.1pt] (0.,3.)-- (3.,6.);
\draw [line width=1.1pt] (5.,6.)-- (8.,6.);
\draw [line width=1.1pt] (5.,6.)-- (5.,3.);
\draw [line width=1.1pt] (5.,3.)-- (8.,3.);
\draw [line width=1.1pt] (8.,6.)-- (8.,3.);
\draw [line width=1.1pt] (5.,6.)-- (8.,3.);
\draw [line width=1.1pt] (5.,3.)-- (8.,6.);
\draw (-6.365736585365862,7.6) node[anchor=north west]{} ;
\draw (1.529365853658537,2.4) node[anchor=north west]{} ;
\begin{scriptsize}
\node [label={[label distance=\dist mm]\NE: {\textcolor{ududff}{$P_1$}}},circle,fill=blue,draw=blue,scale=0.5](A1) at (-5,6) {};
\node [label={[label distance=\distlarger mm]\NW: {\textcolor{ududff}
{$P_1^{\, \sigma}$}}},circle,fill=blue,draw=blue,scale=0.5](A1) at (-2,6) {};
\node [label={[label distance=\dist mm]\SE: {\textcolor{ududff}{$P_2$}}},circle,fill=blue,draw=blue,scale=0.5](A1) at (-5,3) {};
\node [label={[label distance=\distlarger mm]\SW: {\textcolor{ududff}
{$P_2^{\, \sigma}$}}},circle,fill=blue,draw=blue,scale=0.5](A1) at (-2,3) {};
\draw[color=black] (-5.41,4.5) node {$w_{17}$};
\draw[color=black] (-3.55,6.2) node {$w_{85}$};
\draw[color=black] (-1.58,4.5) node {$w_{17}$};
\draw[color=black] (-3.55,2.8) node {$w_{85}$};
\draw[color=black] (-4.51,5.1) node {$w_5$};
\draw[color=black] (-2.5,5.1) node {$w_5$};
\node [label={[label distance=\dist mm]\NE: {\textcolor{ududff}{$P_3$}}},circle,fill=blue,draw=blue,scale=0.5](A1) at (0,6) {};
\node [label={[label distance=\distlarger mm]\NW: {\textcolor{ududff}
{$P_3^{\, \sigma}$}}},circle,fill=blue,draw=blue,scale=0.5](A1) at (3,6) {};
\node [label={[label distance=\dist mm]\SE: {\textcolor{ududff}{$P_4$}}},circle,fill=blue,draw=blue,scale=0.5](A1) at (0,3) {};
\node [label={[label distance=\distlarger mm]\SW: {\textcolor{ududff}
{$P_4^{\, \sigma}$}}},circle,fill=blue,draw=blue,scale=0.5](A1) at (3,3) {};
\draw[color=black] (1.55,6.2) node {$w_{85}$};
\draw[color=black] (1.55,2.8) node {$w_{85}$};
\draw[color=black] (-0.35,4.5) node {$w_{17}$};
\draw[color=black] (3.35,4.5) node {$w_{17}$};
\draw[color=black] (0.5,5.1) node {$w_{5}$};
\draw[color=black] (2.6,5.1) node {$w_{5}$};
\node [label={[label distance=\dist mm]\NE: {\textcolor{ududff}{$P_5$}}},circle,fill=blue,draw=blue,scale=0.5](A1) at (5,6) {};
\node [label={[label distance=\distlarger mm]\NW: {\textcolor{ududff}
{$P_5^{\, \sigma}$}}},circle,fill=blue,draw=blue,scale=0.5](A1) at (8,6) {};
\node [label={[label distance=\dist mm]\SE: {\textcolor{ududff}{$P_6$}}},circle,fill=blue,draw=blue,scale=0.5](A1) at (5,3) {};
\node [label={[label distance=\distlarger mm]\SW: {\textcolor{ududff}
{$P_6^{\, \sigma}$}}},circle,fill=blue,draw=blue,scale=0.5](A1) at (8,3) {};
\draw[color=black] (6.5,6.2) node {$w_{85}$};
\draw[color=black] (4.65,4.5) node {$w_{17}$};
\draw[color=black] (6.5,2.8) node {$w_{85}$};
\draw[color=black] (8.35,4.5) node {$w_{17}$};
\draw[color=black] (5.5,5.1) node {$w_{5}$};
\draw[color=black] (7.5,5.1) node {$w_5$}; 
\end{scriptsize}
\end{tikzpicture}
\end{table}


\begin{table}[!ht]
    \caption{All non-cuspidal quadratic points on \boldmath $X_0(97)$}
    \label{table97}
    \begin{flalign*}
	& \text{Genus: } 7 \\
        & J_0(97)(\Q) \isom \Z^3 \times \Z/8\Z \\
    \end{flalign*} 		
    {\small
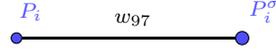
    \begin{tabular}{cccc}
	\toprule
	Point & Field   & $j$-invariant  & CM \\ [0.3ex]
	\midrule
	$P_1$ & $\Q(\sqrt{-3})$  & $54000$ & $-12$ \\ [0.5ex]
        $P_2$ & $\Q(\sqrt{-163})$  & $-262537412640768000$ & $-163$ \\ [0.5ex]
        $P_3$ & $\Q(\sqrt{-1})$ & $1728$ & $-4$ \\ [0.5ex]
        $P_4$ & $\Q(\sqrt{-2})$ & $8000$ & $-8$\\ [0.5ex]
        $P_5$ & $\Q(\sqrt{-43})$ & $-884736000$ & $-43$\\ [0.5ex]
        $P_6$ & $\Q(\sqrt{-11})$ & $-32768$ & $-11$ \\ [0.5ex]
        $P_7$ & $\Q(\sqrt{-3})$ & $0$ & $-3$ \\ [0.5ex]
        $P_8$ & $\Q(\sqrt{-1})$ & $287496$ & $-16$ \\ [0.5ex]
        $P_9$ & $\Q(\sqrt{-3})$ & $-12288000$ & $-27$ \\ [0.5ex] 		
	\bottomrule
    \end{tabular}
    } \\
\begin{tikzpicture}[line cap=round,line join=round,>=triangle 45,x=1.0cm,y=1.0cm]
\clip(-1.1472776948271382,-1.6850082903931902) rectangle (4.271247212060318,1.4596713430682824);
\draw [line width=1.1pt] (0.,0.)-- (3.,0.);
\draw (0.56183619272532,-0.45132628034292016) node[anchor=north west] {$i=1, \dotsc, 9$};
\begin{scriptsize}
\draw [fill=ududff] (0.,0.) circle (2.0pt);
\draw[color=ududff] (0.1892111493939868,0.34693854968960747) node {$P_i$};
\draw [fill=ududff] (3.,0.) circle (2.5pt);
\draw[color=ududff] (3.3,0.35) node {$P_i^{\sigma}$};
\draw[color=black] (1.55,0.24) node {$w_{97}$};
\end{scriptsize}
\end{tikzpicture}
\end{table}


\begin{table}[!ht]
    \caption{All non-cuspidal quadratic points on \boldmath $X_0(98)$}
    \label{table98}
    \begin{align*}
	& \text{Genus: } 7 \\		
        & J_0(98)(\Q) \isom \Z/2\Z \times \Z/6\Z \times \Z/42\Z \\
    \end{align*}
    {\small
    \begin{tabular}{cccc}
	\toprule
	Point & Field  & $j$-invariant & CM  \\ [0.3ex]
	\midrule
        $P_1$ & $\Q(\sqrt{-3})$ & $54000$ & $-12$\\ [0.5ex]
        $P_2$ & $\Q(\sqrt{-3})$ & $0$ & $-3$ \\  [0.5ex] 
        \bottomrule      
    \end{tabular}
    } \\
\begin{tikzpicture}[line cap=round,line join=round,>=triangle 45,x=1.0cm,y=1.0cm]
\clip(-6.85658341463415,1.5842829268292604) rectangle (-0.2600175609756133,7.6354634146341365);
\draw [line width=1.1pt] (-5.,6.)-- (-5.,3.);
\draw [line width=1.1pt] (-5.,6.)-- (-2.,6.);
\draw [line width=1.1pt] (-2.,6.)-- (-2.,3.);
\draw [line width=1.1pt] (-5.,3.)-- (-2.,3.);
\draw [line width=1.1pt] (-5.,6.)-- (-2.,3.);
\draw [line width=1.1pt] (-5.,3.)-- (-2.,6.);
\begin{scriptsize}
\node [label={[label distance=\dist mm]\NE: {\textcolor{ududff}{$P_1$}}},circle,fill=blue,draw=blue,scale=0.5](A1) at (-5,6) {};
\node [label={[label distance=\distlarger mm]\NW: {\textcolor{ududff}
{$P_1^{\, \sigma}$}}},circle,fill=blue,draw=blue,scale=0.5](A1) at (-2,6) {};
\node [label={[label distance=\dist mm]\SE: {\textcolor{ududff}{$P_2$}}},circle,fill=blue,draw=blue,scale=0.5](A1) at (-5,3) {};
\node [label={[label distance=\distlarger mm]\SW: {\textcolor{ududff}
{$P_2^{\, \sigma}$}}},circle,fill=blue,draw=blue,scale=0.5](A1) at (-2,3) {};
\draw[color=black] (-5.3,4.55) node {$w_{2}$};
\draw[color=black] (-3.5,6.2) node {$w_{49}$};
\draw[color=black] (-1.7,4.55) node {$w_2$};
\draw[color=black] (-3.5,2.8) node {$w_{49}$};
\draw[color=black] (-4.5,5.1) node {$w_{98}$};
\draw[color=black] (-2.5,5.1) node {$w_{98}$};
\end{scriptsize}
\end{tikzpicture}
\end{table}


\begin{table}[!ht] 
    \caption{All non-cuspidal quadratic points on \boldmath $X_0(100)$}
    \label{table100}
    \begin{flalign*}
	& \text{Genus: } 7 \\
        & J_0(100)(\Q) \isom \Z/3\Z \times  \Z/30\Z \times  \Z/30\Z  \\
    \end{flalign*} 	
    {\small
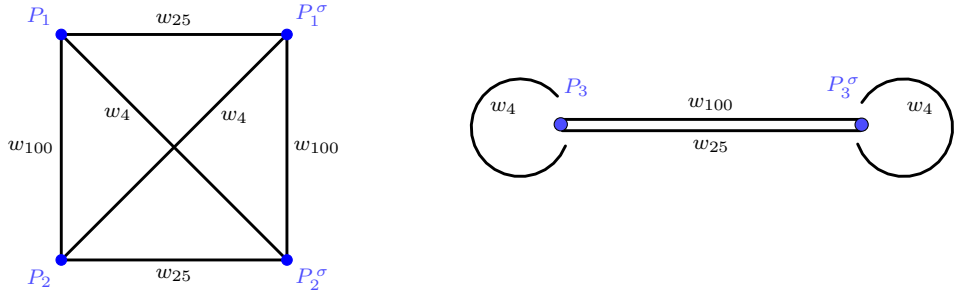
    \begin{tabular}{cccc}
	\toprule
	Point & Field  & $j$-invariant  & CM \\ [0.3ex]
	\midrule
	$P_1$ & $\Q(\sqrt{-1})$  & $1728$ & $-4$ \\ [0.5ex]
        $P_2$ & $\Q(\sqrt{-1})$  & $287496$ & $-16$ \\ [0.5ex]
        $P_3$ & $\Q(\sqrt{-1})$ & $287496$ & $-16$ \\ [0.5ex] 		
	\bottomrule
    \end{tabular}
    } \\
\begin{tikzpicture}[line cap=round,line join=round,>=triangle 45,x=1.0cm,y=1.0cm]
\clip(-6.084311271347847,2.315162519980797) rectangle (6.873260078590276,7.039910758690617);
\draw [line width=1.1pt] (-5.,6.)-- (-5.,3.);
\draw [line width=1.1pt] (-5.,6.)-- (-2.,6.);
\draw [line width=1.1pt] (-2.,6.)-- (-2.,3.);
\draw [line width=1.1pt] (-5.,3.)-- (-2.,3.);
\draw [line width=1.1pt] (-5.,6.)-- (-2.,3.);
\draw [line width=1.1pt] (-5.,3.)-- (-2.,6.);
\draw [line width=1.1pt] (1.66,4.87) -- (5.63,4.87);
\draw [line width=1.1pt] (1.66,4.72) -- (5.63,4.72);
\draw [shift={(1.1,4.76)},line width=1.1pt]  plot[domain=0.7:5.9,variable=\t]({1.*0.65*cos(\t r)+0.*0.65*sin(\t r)},{0.*0.65*cos(\t r)+1.*0.65*sin(\t r)});
\draw [shift={(6.2,4.76)},line width=1.1pt]  plot[domain=-2.8:2.6,variable=\t]({1.*0.65*cos(\t r)+0.*0.65*sin(\t r)},{0.*0.65*cos(\t r)+1.*0.65*sin(\t r)});
\begin{scriptsize}
\node [label={[label distance=\dist mm]\NE: {\textcolor{ududff}{$P_1$}}},circle,fill=blue,draw=blue,scale=0.5](A1) at (-5,6) {};
\node [label={[label distance=\distlarger mm]\NW: {\textcolor{ududff}
{$P_1^{\, \sigma}$}}},circle,fill=blue,draw=blue,scale=0.5](A1) at (-2,6) {};
\node [label={[label distance=\dist mm]\SE: {\textcolor{ududff}{$P_2$}}},circle,fill=blue,draw=blue,scale=0.5](A1) at (-5,3) {};
\node [label={[label distance=\distlarger mm]\SW: {\textcolor{ududff}
{$P_2^{\, \sigma}$}}},circle,fill=blue,draw=blue,scale=0.5](A1) at (-2,3) {};
\draw[color=black] (-5.4,4.5) node {$w_{100}$};
\draw[color=black] (-3.5,6.2) node {$w_{25}$};
\draw[color=black] (-1.6,4.5) node {$w_{100}$};
\draw[color=black] (-3.5,2.8) node {$w_{25}$};
\draw[color=black] (-4.25,4.94) node {$w_{4}$};
\draw[color=black] (-2.7, 4.9) node {$w_{4}$};
\draw [fill=ududff] (1.64,4.804653658536576) circle (2.5pt);
\draw[color=ududff] (1.85,5.3) node {$P_3$};
\draw [fill=ududff] (5.64,4.8) circle (2.5pt);
\draw[color=ududff] (5.4,5.3) node {$P_3^{\, \sigma}$};
\draw[color=black] (3.64,5.07) node {$w_{100}$};
\draw[color=black] (3.64,4.52) node {$w_{25}$};
\draw[color=black] (0.88,5.02) node {$w_4$};
\draw[color=black] (6.42,5.02) node {$w_4$};
\end{scriptsize}
\end{tikzpicture}
\end{table}


\begin{table}[!ht]
    \caption{All non-cuspidal quadratic points on \boldmath $X_0(103)$}
    \label{table103}
    \begin{flalign*}
	& \text{Genus: } 8 \\
        &  J_0(103)(\Q)\isom \Z^2 \times \Z/17\Z \\
    \end{flalign*} 
    {\small
    \begin{tabular}{cccc}
	\toprule
	Point & Field  & $j$-invariant & CM \\ [0.3ex]
	\midrule
        $P_1$ & $\Q(\sqrt{-3})$  & $54000$ & $-12$ \\ [0.5ex]
        $P_2$ & $\Q(\sqrt{-3})$ & $-12288000$ & $-27$ \\ [0.5ex]
        $P_3$ & $\Q(\sqrt{-3})$  & $0$ & $-3$ \\ [0.5ex]
        $P_4$ & $\Q(\sqrt{-11})$  & $-32768$ & $-11$ \\ [0.5ex]
        $P_5$ & $\Q(\sqrt{-67})$  & $-147197952000$ & $-67$ \\ [0.5ex]
        $P_6$ & $\Q(\sqrt{-43})$  & $-884736000$ & $-43$ \\ 
        $P_7$ & $\Q(\sqrt{2885})$  & 
        \begin{tabular}{@{}l@{}}
            \begin{tiny} $-669908635472124980731701532753920\sqrt{2885} $ \end{tiny} \\
            \begin{tiny} \quad $ +35982263935929364331785036841779200$ \end{tiny} \\ 
        \end{tabular}   & NO \\ [0.5ex] 
        \bottomrule
    \end{tabular}
    }\\
\begin{tikzpicture}[line cap=round,line join=round,>=triangle 45,x=1.0cm,y=1.0cm]
\clip(-1.1472776948271382,-1.6850082903931902) rectangle (4.271247212060318,1.4596713430682824);
\draw [line width=1.1pt] (0.,0.)-- (3.,0.);
\draw (0.56183619272532,-0.45132628034292016) node[anchor=north west] {$i=1, \dotsc, 7$};
\begin{scriptsize}
\draw [fill=ududff] (0.,0.) circle (2.0pt);
\draw[color=ududff] (0.1892111493939868,0.34693854968960747) node {$P_i$};
\draw [fill=ududff] (3.,0.) circle (2.5pt);
\draw[color=ududff] (3.3,0.35) node {$P_i^{\sigma}$};
\draw[color=black] (1.55,0.24) node {$w_{103}$};
\end{scriptsize}
\end{tikzpicture}
 \end{table}
 

\begin{table}[!ht]
    \caption{All non-cuspidal quadratic points on \boldmath $X_0(107)$}
    \label{table07}
    \begin{flalign*}
	& \text{Genus: } 9 \\
        &  J_0(107)(\Q)\isom \Z^2 \times \Z/53\Z \\
    \end{flalign*}
    {\small
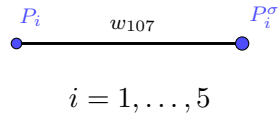
    \begin{tabular}{cccc}
	\toprule
	Point & Field & $j$-invariant & CM \\ [0.3ex]
	\midrule
        $P_1$ & $\Q(\sqrt{-7})$  & $-3375$ & $-7$ \\ [0.5ex]
        $P_2$ & $\Q(\sqrt{-7})$  & $16581375$ & $-28$ \\ [0.5ex]
        $P_3$ & $\Q(\sqrt{-2})$ & $8000$ & $-8$ \\ [0.5ex]
        $P_4$ & $\Q(\sqrt{-43})$  & $-884736000$ & $-43$ \\ [0.5ex]
        $P_5$ & $\Q(\sqrt{-67})$ & $-147197952000$ & $-67$\\ [0.5ex] 
        \bottomrule
  \end{tabular}
  } \\
\begin{tikzpicture}[line cap=round,line join=round,>=triangle 45,x=1.0cm,y=1.0cm]
\clip(-1.1472776948271382,-1.6850082903931902) rectangle (4.271247212060318,1.4596713430682824);
\draw [line width=1.1pt] (0.,0.)-- (3.,0.);
\draw (0.56183619272532,-0.45132628034292016) node[anchor=north west] {$i=1, \dotsc, 5$};
\begin{scriptsize}
\draw [fill=ududff] (0.,0.) circle (2.0pt);
\draw[color=ududff] (0.1892111493939868,0.34693854968960747) node {$P_i$};
\draw [fill=ududff] (3.,0.) circle (2.5pt);
\draw[color=ududff] (3.3,0.35) node {$P_i^{\sigma}$};
\draw[color=black] (1.55,0.24) node {$w_{107}$};
\end{scriptsize}
\end{tikzpicture}
\end{table}


\begin{table}[!ht]
    \caption{All non-cuspidal quadratic points on \boldmath $X_0(109)$}
    \label{table109}
    \begin{flalign*}
	& \text{Genus: } 8 \\
        & J_0(109)(\Q) \isom \Z^3 \times \Z/9\Z\\
    \end{flalign*} 
    {\small
    \begin{tabular}{cccc}
	\toprule
	Point & Field  & $j$-invariant  & CM \\ [0.3ex]
	\midrule
        $P_1$ & $\Q(\sqrt{-43})$  & $-884736000$ & $-43$\\ [0.5ex]
        $P_2$ & $\Q(\sqrt{-3})$ & $54000$ & $-12$ \\ [0.5ex]
        $P_3$ & $\Q(\sqrt{-1})$ & $1728$ & $-4$\\ [0.5ex]
        $P_4$ & $\Q(\sqrt{-7})$ & $16581375$ & $-28$\\ [0.5ex]
        $P_5$ & $\Q(\sqrt{-3})$ & $-12288000$ & $-27$\\ [0.5ex]
        $P_6$ & $\Q(\sqrt{-7})$ & $-3375$ & $-7$\\ [0.5ex]
        $P_7$ & $\Q(\sqrt{-3})$ & $0$ & $-3$\\ [0.5ex]
        $P_8$ & $\Q(\sqrt{-1})$ & $287496$ & $-16$ \\	[0.5ex] 	
	\bottomrule
    \end{tabular}
    } \\
\begin{tikzpicture}[line cap=round,line join=round,>=triangle 45,x=1.0cm,y=1.0cm]
\clip(-1.1472776948271382,-1.6850082903931902) rectangle (4.271247212060318,1.4596713430682824);
\draw [line width=1.1pt] (0.,0.)-- (3.,0.);
\draw (0.56183619272532,-0.45132628034292016) node[anchor=north west] {$i=1, \dotsc, 8$};
\begin{scriptsize}
\draw [fill=ududff] (0.,0.) circle (2.0pt);
\draw[color=ududff] (0.1892111493939868,0.34693854968960747) node {$P_i$};
\draw [fill=ududff] (3.,0.) circle (2.5pt);
\draw[color=ududff] (3.3,0.35) node {$P_i^{\sigma}$};
\draw[color=black] (1.55,0.24) node {$w_{109}$};
\end{scriptsize}
\end{tikzpicture}
\end{table}


\begin{table}[!ht]
    \caption{All non-cuspidal quadratic points on \boldmath $X_0(113)$}
    \label{table113}
    \begin{flalign*}
	& \text{Genus: } 9 \\
        & J_0(127)(\Q) \isom \Z^3 \times \Z/28\Z \\
    \end{flalign*}
    {\small
    \begin{tabular}{cccc}
	\toprule
	Point & Field & $j$-invariant & CM  \\ [0.3ex]
	\midrule
        $P_1$ & $\Q(\sqrt{-1})$ & $287496$ & $-16$ \\ [0.5ex]
        $P_2$ & $\Q(\sqrt{-1})$  & $1728$ & $-4$ \\ [0.5ex]
        $P_3$ & $\Q(\sqrt{-7})$  & $16581375$ & $-28$ \\ [0.5ex]
        $P_4$ & $\Q(\sqrt{-7})$  & $-3375$ & $-7$ \\ [0.5ex]
        $P_5$ & $\Q(\sqrt{-11})$  & $-32768$ & $-11$\\ [0.5ex]
        $P_6$ & $\Q(\sqrt{-163})$ & $-262537412640768000$ & $-163$ \\ [0.5ex]
        $P_7$ & $\Q(\sqrt{-2})$  & $8000$ & $-8$ \\ [0.5ex] 
        \bottomrule
    \end{tabular}
    } \\
\begin{tikzpicture}[line cap=round,line join=round,>=triangle 45,x=1.0cm,y=1.0cm]
\clip(-1.1472776948271382,-1.6850082903931902) rectangle (4.271247212060318,1.4596713430682824);
\draw [line width=1.1pt] (0.,0.)-- (3.,0.);
\draw (0.56183619272532,-0.45132628034292016) node[anchor=north west] {$i=1, \dotsc, 7$};
\begin{scriptsize}
\draw [fill=ududff] (0.,0.) circle (2.0pt);
\draw[color=ududff] (0.1892111493939868,0.34693854968960747) node {$P_i$};
\draw [fill=ududff] (3.,0.) circle (2.5pt);
\draw[color=ududff] (3.3,0.35) node {$P_i^{\sigma}$};
\draw[color=black] (1.55,0.24) node {$w_{113}$};
\end{scriptsize}
\end{tikzpicture}
\end{table}


\begin{table}[!ht]
    \caption{All non-cuspidal quadratic points on \boldmath $X_0(121)$}
    \label{table121}
    \begin{flalign*}
	& \text{Genus: 6 } \\
        & J_0(121)(\Q) \isom \Z \times \Z/5\Z \times \Z/5\Z \\
    \end{flalign*}
    {\small
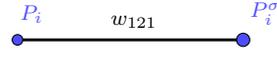
    \begin{tabular}{ccccc}
	\toprule
	Point & Field  & $j$-invariant & CM  \\ [0.3ex]
	\midrule 
        $P_1$ & $\Q(\sqrt{-19})$  & $-884736$ & $-19$ \\ [0.5ex]
        $P_2$ & $\Q(\sqrt{-43})$  & $-884736000$ & $-43$\\  [0.5ex]
        $P_3$ & $\Q(\sqrt{-2})$ & $8000$ & $-8$  \\ [0.5ex]
        $P_4$ & $\Q(\sqrt{-7})$ & $-3375$ & $-7$ \\ [0.5ex] 
        $P_5$ & $\Q(\sqrt{-7})$ & $16581375$ & $-28$ \\ [0.5ex] 
        \bottomrule
    \end{tabular} 
    } \\          
\begin{tikzpicture}[line cap=round,line join=round,>=triangle 45,x=1.0cm,y=1.0cm]
\clip(-1.1472776948271382,-1.6850082903931902) rectangle (4.271247212060318,0.9);
\draw [line width=1.1pt] (0.,0.)-- (3.,0.);
\draw (0.56183619272532,-0.45132628034292016) node[anchor=north west] {$i=1, \dotsc, 5$};
\begin{scriptsize}
\draw [fill=ududff] (0.,0.) circle (2.0pt);
\draw[color=ududff] (0.1892111493939868,0.34693854968960747) node {$P_i$};
\draw [fill=ududff] (3.,0.) circle (2.5pt);
\draw[color=ududff] (3.3,0.35) node {$P_i^{\sigma}$};
\draw[color=black] (1.55,0.24) node {$w_{121}$};
\end{scriptsize}
\end{tikzpicture}
\end{table}    


\begin{table}[!ht]
    \caption{All non-cuspidal quadratic points on \boldmath $X_0(127)$}
    \label{table127}
    \begin{flalign*}
	& \text{Genus: } 10 \\
        &  J_0(127)(\Q) \isom \Z^3 \times \Z/21\Z \\
    \end{flalign*}
    {\small
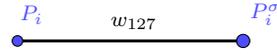
    \begin{tabular}{cccc}
	\toprule
	Point & Field  & $j$-invariant & CM  \\ [0.3ex]
	\midrule 
        $P_1$ & $\Q(\sqrt{-67})$ & $ -147197952000$ & $-67$ \\ [0.5ex]
        $P_2$ & $\Q(\sqrt{-3})$ & $54000$ & $-12$ \\ [0.5ex]
        $P_3$ & $\Q(\sqrt{-3})$ & $-12288000$ & $-27$ \\ [0.5ex]
        $P_4$ & $\Q(\sqrt{-7})$ & $16581375$ & $-28$ \\ [0.5ex]
        $P_5$ & $\Q(\sqrt{-7})$ & $-3375$ & $-7$\\ [0.5ex]
        $P_6$ & $\Q(\sqrt{-3})$  & $0$ & $-3$ \\ [0.5ex]
        $P_7$ & $\Q(\sqrt{-43})$  & $-884736000$ & $-43$ \\ [0.5ex] 
        \bottomrule
    \end{tabular}
    } \\
\begin{tikzpicture}[line cap=round,line join=round,>=triangle 45,x=1.0cm,y=1.0cm]
\tikzmath{\x1 = 0.35; \y1 =-0.05; \z1=180; \w1=0.2;} 
\small

\clip(-1.1472776948271382,-1.6850082903931902) rectangle (4.271247212060318,1.4596713430682824);
\draw [line width=1.1pt] (0.,0.)-- (3.,0.);
\draw (0.56183619272532,-0.45132628034292016) node[anchor=north west] {$i=1, \dotsc, 7$};
\begin{scriptsize}
\draw [fill=ududff] (0.,0.) circle (2.0pt);
\draw[color=ududff] (0.1892111493939868,0.34693854968960747) node {$P_i$};
\draw [fill=ududff] (3.,0.) circle (2.5pt);
\draw[color=ududff] (3.3,0.35) node {$P_i^{\sigma}$};
\draw[color=black] (1.55,0.24) node {$w_{127}$};
\end{scriptsize}
\end{tikzpicture}
\end{table}


\clearpage
\newpage


\bibliographystyle{siam}
\bibliography{references}

\begin{thebibliography}{10}

\bibitem{AABCCKW}
{\sc N.~Ad\v{z}aga, V.~Arul, L.~Beneish, M.~Chen, S.~Chidambaram, T.~Keller,
  and B.~Wen}, {\em {Quadratic Chabauty for Atkin-Lehner Quotients of Modular
  Curves of Prime Level and Genus 4, 5, 6}}, Acta Arith., 208 (2023),
  pp.~15--49.

\bibitem{ACKP22}
{\sc N.~Ad\v{z}aga, S.~Chidambaram, T.~Keller, and O.~Padurariu}, {\em Rational
  points on hyperelliptic {A}tkin-{L}ehner quotients of modular curves and
  their coverings}, Res. Number Theory, 8 (2022).
\newblock Paper No. 87.

\bibitem{BBB}
{\sc J.~S. Balakrishnan, A.~J. Best, F.~Bianchi, B.~Lawrence, J.~S. M\"{u}ller,
  N.~Triantafillou, and J.~Vonk}, {\em Two recent {$p$}-adic approaches towards
  the (effective) {M}ordell conjecture}, in Arithmetic {L}-functions and
  differential geometric methods, vol.~338 of Progr. Math.,
  Birkh\"{a}user/Springer, Cham, 2021, pp.~31--74.

\bibitem{BDMTV2}
{\sc J.~S. Balakrishnan, N.~Dogra, J.~S. M\"{u}ller, J.~Tuitman, and J.~Vonk},
  {\em Quadratic {C}habauty for modular curves: Algorithms and examples},
  Compos. Math., 159 (2023), pp.~1111--1152.

\bibitem{BalTui}
{\sc J.~S. Balakrishnan and J.~Tuitman}, {\em Explicit {C}oleman integration
  for curves}, Math. Comp., 89 (2020), pp.~2965--2984.

\bibitem{balcik}
{\sc I.~Balçık}, {\em Growth of odd torsion over imaginary quadratic fields
  of class number 1}, 2021.
\newblock Available at \url{https://arxiv.org/abs/2111.11399}.

\bibitem{BPN}
{\sc B.~S. Banwait, F.~Najman, and O.~Padurariu}, {\em Cyclic isogenies of
  elliptic curves over fixed quadratic fields}, 2022.
\newblock To appear in Math. Comp., available at
  \url{https://arxiv.org/abs/2206.08891}.

\bibitem{hyperQ}
{\sc F.~Bars, J.~Gonz\'{a}lez, and X.~Xarles}, {\em Hyperelliptic
  parametrizations of {$\mathbb{Q}$} curves}, Ramanujan J., 56 (2021),
  pp.~103--120.

\bibitem{magma}
{\sc W.~Bosma, J.~Cannon, and C.~Playoust}, {\em The {M}agma algebra system.
  {I}. {T}he user language}, J. Symbolic Comput., 24 (1997), pp.~235--265.

\bibitem{box21}
{\sc J.~Box}, {\em Quadratic points on modular curves with infinite
  {M}ordell--{W}eil group}, Math. Comp., 90 (2021), pp.~321--343.

\bibitem{BruinStoll2010}
{\sc N.~Bruin and M.~Stoll}, {\em The {M}ordell--{W}eil sieve: proving
  non-existence of rational points on curves}, LMS J. Comput. Math., 13 (2010),
  pp.~272--306.

\bibitem{bruin-najman}
{\sc P.~Bruin and F.~Najman}, {\em Hyperelliptic modular curves {$X_0(n)$} and
  isogenies of elliptic curves over quadratic fields}, LMS J. Comput. Math., 18
  (2015), pp.~578--602.

\bibitem{CGPS21}
{\sc P.~L. Clark, T.~Genao, P.~Pollack, and F.~Saia}, {\em The least degree of
  a {CM} point on a modular curve}, {J. London. Math. Soc.}, 105 (2022),
  pp.~825--883.

\bibitem{Deg3Class}
{\sc M.~{Derickx}, A.~{Etropolski}, M.~{van Hoeij}, J.~S. {Morrow}, and
  D.~{Zureick-Brown}}, {\em {Sporadic cubic torsion}}, {Algebra Number Theory},
  15 (2021), pp.~1837--1864.

\bibitem{DKSS}
{\sc M.~Derickx, S.~Kamienny, W.~Stein, and M.~Stoll}, {\em Torsion points on
  elliptic curves over number fields of small degree}, Algebra Number Theory,
  17 (2023), pp.~267--308.

\bibitem{Elkies2004}
{\sc N.~D. {Elkies}}, {\em {On elliptic $K$-curves}}, in Modular curves and
  Abelian varieties. Based on lectures of the conference, Bellaterra,
  Barcelona, July 15--18, 2002, Basel: Birkh\"auser, 2004, pp.~81--91.

\bibitem{Freitas-Siksek}
{\sc N.~Freitas and S.~Siksek}, {\em Fermat's last theorem over some small real
  quadratic fields}, Algebra Number Theory, 9 (2015), pp.~875--895.

\bibitem{galbraith}
{\sc S.~Galbraith}, {\em Equations for modular curves}, PhD thesis, University
  of Oxford, 1996.

\bibitem{GrossZagier1986}
{\sc B.~H. Gross and D.~B. Zagier}, {\em Heegner points and derivatives of
  {$L$}-series}, Invent. Math., 84 (1986), pp.~225--320.

\bibitem{HarrisSilverman}
{\sc J.~Harris and J.~H. Silverman}, {\em Bielliptic curves and symmetric
  products}, Proc. Am. Math. Soc., 112 (1991), pp.~347--356.

\bibitem{kamienny92}
{\sc S.~Kamienny}, {\em Torsion points on elliptic curves and
  {$q$}-coefficients of modular forms}, Invent. Math., 109 (1992),
  pp.~221--229.

\bibitem{kenku:125}
{\sc M.~A. Kenku}, {\em On the modular curves {$X_{0}(125)$}, {$X_{1}(25)$} and
  {$X_{1}(49)$}}, J. London Math. Soc., 23 (1981), pp.~415--427.

\bibitem{KM88}
{\sc M.~A. Kenku and F.~Momose}, {\em Torsion points on elliptic curves defined
  over quadratic fields}, Nagoya Math. J., 109 (1988), pp.~125--149.

\bibitem{Khawaja-Jarvis}
{\sc M.~Khawaja and F.~Jarvis}, {\em Fermat's {L}ast {T}heorem over
  $\mathbb{Q}(\sqrt{2}, \sqrt{3})$}, 2022.
\newblock Available at \url{https://arxiv.org/abs/2210.03744}.

\bibitem{KolyvaginLogachev}
{\sc V.~A. Kolyvagin and D.~Y. Logach\"{e}v}, {\em Finiteness of the
  {S}hafarevich-{T}ate group and the group of rational points for some modular
  abelian varieties}, Algebra i Analiz, 1 (1989), pp.~171--196.

\bibitem{eisenstein}
{\sc B.~Mazur}, {\em Modular curves and the {E}isenstein ideal}, Inst. Hautes
  \'{E}tudes Sci. Publ. Math., 47 (1977), pp.~33--186 (1978).
\newblock With an appendix by Mazur and M. Rapoport.

\bibitem{mazur:isogenies}
\leavevmode\vrule height 2pt depth -1.6pt width 23pt, {\em Rational isogenies
  of prime degree}, Invent. Math., 44 (1978), pp.~129--162.
\newblock With an appendix by {D}. {G}oldfeld.

\bibitem{merel}
{\sc L.~Merel}, {\em Bornes pour la torsion des courbes elliptiques sur les
  corps de nombres}, Invent. Math., 124 (1996), pp.~437--449.

\bibitem{Michaud-Jacobs}
{\sc P.~Michaud-Jacobs}, {\em Fermat's {L}ast {T}heorem and modular curves over
  real quadratic fields}, Acta Arith., 203 (2022), pp.~319--351.

\bibitem{MJ_cart}
{\sc P.~Michaud-Rodgers (Michaud-Jacobs)}, {\em {Quadratic points on non-split
  Cartan modular curves}}, Int. J. Number Theory, 18 (2022), pp.~245--267.

\bibitem{miranda}
{\sc R.~Miranda}, {\em {Algebraic Curves and Riemann Surfaces}}, vol.~5 of
  Graduate Studies in Mathematics, American Mathematical Society, Providence,
  RI, 1995.

\bibitem{Momose1987}
{\sc F.~Momose}, {\em {Rational points on the modular curves $X^+_0(N)$}}, {J.
  Math. Soc. Japan}, 39 (1987), pp.~269--286.

\bibitem{najman-vukorepa}
{\sc F.~Najman and B.~Vukorepa}, {\em Quadratic points on bielliptic modular
  curves}, Math. Comp., 92 (2023), pp.~1791--1816.

\bibitem{ozman-siksek}
{\sc E.~Ozman and S.~Siksek}, {\em Quadratic points on modular curves}, Math.
  Comp., 88 (2019), pp.~2461--2484.

\bibitem{shafarevich}
{\sc I.~Shafarevich}, {\em {Basic Algebraic Geometry 1}}, Springer Berlin,
  Heidelberg, 3rd~ed., 2013.

\bibitem{siksek:symchabauty}
{\sc S.~Siksek}, {\em Chabauty for symmetric powers of curves}, Algebra Number
  Theory, 3 (2009), pp.~209--236.

\bibitem{stacks-project}
{\sc {The Stacks project authors}}, {\em {The Stacks project}}.
\newblock \url{https://stacks.math.columbia.edu}, 2023.
\newblock (Accessed 26 January 2023).

\bibitem{Vux91}
{\sc B.~Vukorepa}, {\em Isogenies over quadratic fields of elliptic curves with
  rational $j$-invariant}, 2022.
\newblock Available at \url{https://arxiv.org/abs/2203.10672}.

\end{thebibliography}

\end{document}